\documentclass[11pt]{article}
\usepackage[hypertexnames=false]{hyperref}
\usepackage{amsmath}
\usepackage{amssymb}
\usepackage{amsthm}
\usepackage{authblk}
\usepackage[margin=2.5cm]{geometry}
\usepackage{bbm}
\usepackage{amsthm}
\usepackage{mathtools}
\usepackage[mathscr]{euscript}
\usepackage[utf8]{inputenc}
\usepackage{color}
\usepackage{enumerate}  
\usepackage{bm}

\usepackage{graphicx}

\newlength{\bibitemsep}
\newlength{\bibparskip}
\let\oldthebibliography\thebibliography
\renewcommand\thebibliography[1]{%
  \oldthebibliography{#1}%
  \setlength{\parskip}{\bibitemsep}%
  \setlength{\itemsep}{\bibparskip}%
}

\newcommand{\R}[0]{\mathbb{R}}

\newcommand{\Z}[0]{\mathbb{Z}}
\newcommand{\N}[0]{\mathbb{N}}

\renewcommand{\P}[0]{\mathbb{P}}
\newcommand{\E}[0]{\mathbb{E}}

\newcommand{\sC}[0]{\mathcal{C}}

\newcommand{\sN}[0]{\mathcal{N}}

\newcommand{\sU}[0]{\mathcal{U}}
\newcommand{\sV}[0]{\mathcal{V}}
\newcommand{\sL}[0]{\mathcal{L}}

\newcommand{\Var}[0]{\text{Var\,}}
\newcommand{\Cov}[0]{\text{Cov\,}}

\newtheorem{theorem}{Theorem}

\newtheorem{lemma}{Lemma}
\newtheorem{prop}{Proposition}

\newtheorem{remark}{Remark}
\newtheorem{condition}{Condition}

\allowdisplaybreaks
\numberwithin{equation}{section}
\numberwithin{theorem}{section}
\numberwithin{lemma}{section}
\numberwithin{corollary}{section}
\numberwithin{prop}{section}
\numberwithin{remark}{section}
\numberwithin{condition}{section}

\newcommand{\eps}{\varepsilon}

\begin{document}

\title{Gumbel laws in the symmetric exclusion process}

\author{Michael Conroy\thanks{Mathematics, University of Arizona, Tucson, AZ 85721 USA michaelconroy@math.arizona.edu}\ \   and\, Sunder Sethuraman\thanks{Mathematics, University of Arizona, Tucson, AZ 85721 USA sethuram@math.arizona.edu}}

\date{}

\maketitle

\begin{abstract}

We consider the symmetric exclusion particle system on $\Z$ starting from an infinite particle step configuration in which there are no particles to the right of a maximal one.  We show that the scaled position $X_t/(\sigma b_t) - a_t$ of the right-most particle at time $t$ converges to a Gumbel limit law, where $b_t = \sqrt{t/\log t}$, $a_t = \log(t/(\sqrt{2\pi}\log t))$, and $\sigma$ is the standard deviation of the random walk jump probabilities.  This work solves a problem left open in Arratia (1983).

Moreover, to investigate the influence of the mass of particles behind the leading one, we consider initial profiles consisting of a block of $L$ particles, where $L \to \infty$ as $t \to \infty$. Gumbel limit laws, under appropriate scaling, are obtained for $X_t$ when $L$ diverges in $t$. In particular, there is a transition when $L$ is of order $b_t$,
above which the displacement of $X_t$ is similar to that under a infinite particle step profile, and below which it is of order $\sqrt{t\log L}$.  

Proofs are based on recently developed negative dependence properties of the symmetric exclusion system.
Remarks are also made on the behavior of the right-most particle starting from a step profile in asymmetric nearest-neighbor exclusion, which complement known results.

\end{abstract}

\smallskip
{\small

}

\section{Introduction}

We consider the symmetric exclusion process on $\Z$ with irreducible, translation invariant transition probabilities 
Let $\{p_i\}$ be irreducible, translation-invariant transition probabilities,
\[
	p_i := p(x, x + i) = p(x, x-i), \qquad x, i\in\Z,  
\]
such that $\sum_i p_i = 1$. 
Informally, the system consists of a typically infinite number of particles that move individually as continuous-time random walks with jump rates $\{p_i\}$ subject to the interaction that jumps to already occupied sites are suppressed. 
Such a process, in the case of nearest-neighbor interaction, models single-file diffusion of particles whose movement is limited by its neighbors. There are also interpretations of the process in terms of queues, fluid flows, traffic, and other phenomena. For general discussions, including its history, see \cite{DeMPre, KipLan99, Lig85, Lignew, Spo91}.

More formally, the process $\{\eta_t: t \ge0\}$ takes values in $\{0,1\}^\Z$, following the unlabled evolution of the particles where, at time $t$, $\eta_t(x)=1$ if $x$ is occupied and $\eta_t(x)=0$ if $x$ is vacant.  The system has Markov generator $\sL$ given by 
\begin{equation}\label{eq:Exgen}
	\sL f(\eta) = \sum_{x,y \in \Z} p_{y-x} \eta(x)(1-\eta(y))
	\left[f(\eta^{x,y}) - f(\eta)\right], 
\end{equation}
for functions $f$ that depend on $\eta(x)$ for only finitely many $x$, 
where 
\[
	\eta^{x,y}(z) = \left\{ \begin{array}{ll} \eta(y) & \;\mbox{if}\; z = x, \\
								\eta(x) & \;\mbox{if}\; z = y, \\
								\eta(z) & \;\mbox{if}\; z \ne x,y.  \end{array} \right. 
\]

We introduce the initial conditions of interest as follows. Consider a Bernoulli product measure
\begin{equation}\label{eq:BernIC}
	\nu(\eta(x) = 1) = \rho_x, \qquad \eta \in \{0,1\}^\Z, \quad x \in \Z, 
\end{equation}
for some $\{\rho_x\} \subset [0,1]$ with $\rho_0 = 1$ and $\rho_x = 0$ for $x > 0$. Since all particles initially start to the left of the origin, this type of initial condition is referred to as a ``step.''  When
$\rho_x = 1$ for all $x \le 0$, we will call this deterministic initial profile as a ``full-step'' profile.  

Note that since $\rho_0=1$, a particle is guaranteed to be initially at the origin. 
We study the behavior of the right-most or maximum particle, which has position at time $t$ given by
\[
	X_t =\max\{ x : \eta_t(x) = 1\}, 
\]
and for which $X_0 = 0$. In the case of nearest-neighbor interaction $p_{\pm 1} = 1/2$, $X_t$ tracks the movement of the particle initially at the origin, since no two particles can change order. Such a particle is referred to as a tagged or tracer particle.

\subsection{Main problem and context}
In Arratia \cite[Theorem 3]{Arr83}, to compare with the motion of a tagged particle in exclusion,
it was shown that when $X_t$ is the right-most position of a system of {\em independent}, symmetric, nearest-neighbor random walks with initially a particle at every integer $k \le 0$, then
\begin{equation}\label{eq:RWGumbel}
	\lim_{t\to\infty} P\left( \frac{X_t}{b_t} - a_t \le x\right) = e^{-e^{-x}}
\end{equation}
for all $x\in \R$, where 
\begin{equation}\label{eq:ArrScaling}
	a_t = \log \left(\frac{t}{\sqrt{2\pi}\log t}\right) \qquad\mbox{and}\qquad b_t = \sqrt{\frac{t}{\log t}};
\end{equation}
see also \cite{Mikosch} in this context.
A version of \eqref{eq:RWGumbel} for single-file diffusions, written as the limit of the maximum of independent Brownian motions, is also known \cite{Sab07}.

A question left open
in \cite{Arr83} is whether the same result holds when $X_t$ is the position of the tracer particle with $X_0=0$
in the simple exclusion system with the same step initial condition. This is motivated in part by the results in \cite{Arr83} that in both cases, $X_t$ has the scaled law of large numbers limit 
\begin{equation}
\label{Arratia-LLN}
	\frac{X_t}{\sqrt{t}} - \sqrt{\log t} \to 0 \qquad\mbox{a.s.,}\qquad t\to\infty, 
\end{equation}
and 
also that $\{b_t^{-1}X_t - a_t: t>0\}$ is tight for the exclusion system. 
It was also noted that this scaling may be inferred through a large deviation analysis in \cite{ImaMalSas17}. However, establishing that the limit in \eqref{eq:RWGumbel} holds for the exclusion system, consisting of infinite interacting random walks on $\Z$, has remained open. Among our results is confirmation that the tagged particle obeys \eqref{eq:RWGumbel}.

Much of the interest and difficulty of this question is that it considers a strongly out-of-equilibrium initial profile, which blocks left jumps of the tagged particle, forcing anomalous motion to the right.  Indeed, the tagged particle has no interaction on one side and under the full-step initial profile cannot go left of the origin.
In particular, by \eqref{Arratia-LLN}, it moves in time $t$ in a space scale $\sqrt{t\log t}$, much greater than the diffusive $\sqrt{t}$ scale of an isolated random walk or the hydrodynamic evolution of the ``bulk'' mass density; see Section VIII.5 in \cite{Lig85}.  

 The advent of the Gumbel law is in contrast to other known limit theorems for a tagged particle in the exclusion particle system on $\Z$.
 In random matrix theory, however, we comment that Gumbel limits have been derived, such as for the edge behavior of eigenvalues in certain ensembles \cite{Joh2007}, which might be seen as a type of continuous space exclusion process.

In the non-degenerate situation of a tracer surrounded by a density of particles, Arratia \cite{Arr83} showed that sub-diffusive $t^{1/4}$ scaling holds for nearest-neighbor interaction. In particular, when the initial product measure is given by $\nu(\eta(x) = 1) = \rho$ for all $x \ne 0$ in $\Z$ and $\nu(\eta(0)=1)=1$, the position $X_t$ of the particle beginning at the origin satisfies
\[
	t^{-1/4} X_t \Rightarrow \sN\left(0, \sqrt{\frac{2}{\pi}} \frac{1-\rho}{\rho}\right), \qquad t \to \infty. 
\]
This result was updated to a fractional Brownian motion, with Hurst parameter ${1}/{4}$, process limit in \cite{PelSet08}; see also \cite{JG} for an extension to a variable diffusion system.

Further, a central limit theorem for the symmetric, nearest-neighbor case on $\Z$ was proved in \cite{JarLan06}, starting from a Bernoulli product measure $\nu^N(\eta(x) = 1) = \rho_0(x/N)$ for a profile $\rho_0: \R \to [0,1]$ with $4$ bounded derivatives. It is shown, in various senses, that as $N \to \infty$, 
\[	
	\frac{X_{Nt}}{\sqrt{N}} \rightarrow  u_t\qquad \mbox{and}\qquad \frac{X_{Nt} - \sqrt{N}u_t}{N^{1/4}} \Rightarrow W_t, 
\]
where $W_t$ is a Gaussian process with a certain covariance structure and $u_t$ satisfies
$\int_0^\infty \big(\rho(t, u) - \rho_0(u)\big)du = \int_0^{u_t} \rho(t, u)du$ with $\rho(t,u)$ the solution of a heat equation with initial condition $\rho_0$.  We observe if $\rho_0(u)=0$ for $u>0$, then $u_t\equiv \infty$, in line with Arratia's law of large numbers \eqref{Arratia-LLN}. The result in \cite{JarLan06} is nontrivial only if $\rho_0$ corresponds to a density of particles on both sides of the origin, excluding the step profiles we consider here.

We remark that large deviation results have been established for the tagged particle in one-dimensional symmetric nearest-neighbor exclusion under certain local-equilibrium initial profiles in \cite{SetVar13}.   
In \cite{ImaMalSas17}, as mentioned earlier, an explicit large deviation function was found for the tagged particle starting from the step initial measure $\nu(\eta(x) = 1) = \rho_-$ for $x \le 0$ and $\nu(\eta(x) = 1) = \rho_+$ for $x > 0$ and $\rho_-\ne \rho_+$. 

In contrast, in the non-nearest-neighbor finite-range case, the tagged particle under stationarity obeys diffusive scaling:
In \cite{KipVar86}, for symmetric, non-nearest-neighbor finite-range exclusion starting from a stationary Bernoulli initial profile $\nu(\eta(x) = 1) = \rho$ for all $x\in \Z$, 
it is shown that the position $X_t$ of a tagged particle satisfies 
\begin{equation}\label{eq:kipvar}
	\frac{X_{N t} - E[X_{N t}]}{\sqrt{N}} \Rightarrow \kappa B_t, \qquad N \to \infty, 
\end{equation}
a functional central limit theorem, where $B_t$ is a Brownian motion and $\kappa^2$ is a certain ``self-diffusion'' coefficient; such a result also holds more generally for finite-range symmetric exclusion $\Z^d$.

We mention in passing that large deviation results in the non-nearest neighbor finite-range and higher dimensional setting have been considered in \cite{QRV}.  See also \cite{Jara}, for stable law limits for the tagged particle in symmetric exclusion with long-range jump rates.  Results for a driven tagged particle in symmetric exclusion on $\Z$ include \cite{Lan-Vol-Oll}, \cite{Wang2}.

\subsection{Discussion of results}

We will consider a general periodic step initial profile where $\rho_{x-m} = \rho_x$ for some $m \in \N_+=\{1,2,\ldots\}$ and every $x \le -1$, where at least one of $\rho_{-1}, \ldots, \rho_{-m}$ is positive (see Condition \ref{ic}). Such profiles correspond to a variety of deterministic (when $\rho_x = 0$ or $1$) as well as random initial conditions (when $0<\rho_x<1$). 

One reason to consider such profiles is that there will be an infinite number of particles in the system, which is necessary for our results to hold (see comments below \eqref{eq:firststir}).  Indeed, with only a finite number of particles in the system, diffusive scaling can be inferred for their positions, a different category not considered in this paper; see, in the nearest-neighbor setting, \cite{Aslangul}, \cite{Lizana-Am}, and \cite{Rodenbeck}.

 Another reason to consider a Bernoulli product initial condition is that the subsequent exclusion evolution will be ``strong Rayleigh'' as discussed in Subsection \ref{neg_dep_subsect}.  Further, one more reason is that such initial profiles allow for some computations, though involved (see Sections \ref{meancomp} and \ref{covcomp}). 
However, our results are robust in that perturbations from periodicity do not affect their conclusions; see Remark \ref{robust}.

Also, we will assume that the jump distribution $\{p_i\}$ has a finite moment generating function $\sum_ie^{\theta i}p_i<\infty$ for some $\theta>0$ (Condition \ref{mgf}).   This condition certainly allows for finite-range interactions as well as infinite-range ones. 

Now let $\sigma^2$ be the variance of $\{p_i\}$, and recall $a_t$ and $b_t$ in \eqref{eq:ArrScaling}.
We show in Theorem \ref{main} that under these assumptions, the scaled position 
\begin{align}
\label{gen-scaling}
\frac{X_t}{\sigma b_t} - a_t
\end{align} converges in distribution to a Gumbel law that reflects the initial condition and jump distribution, as $t \to \infty$.  As noted in Remark \ref{mth}, more generally, the $m$th order statistic of the process converges weakly to a related law involving a certain Poisson distribution. We comment that the limit \eqref{eq:RWGumbel} for the tagged particle in the nearest-neighbor setting when $\sigma^2=1$ is a case of these results.

A second goal of this work is to investigate a natural question that arises: how sensitive is the leading particle behavior to the size of the block of particles behind it?  In other words, how large should the size of the block be to force the leading particle to move in anomalous scale to the right?
To probe this sensitivity, we consider a sequence of exclusion processes with (time-dependent) periodic Bernoulli ``$L$-step'' initial conditions $\nu_L$, where $\nu_L(\eta(x) = 1) = 0$ for $x \notin \{-L, \ldots, 0\}$ for $L = L(t) \to \infty$ as $t \to\infty$. The rate at which $L \to \infty$ determines the limiting behavior of $X_t$ as $t \to \infty$. 

In particular, the same Gumbel limit as in the full step case holds when $L \sqrt{\log t/t} \to \infty$. 
However, we derive a different Gumbel law for \eqref{gen-scaling} in the same scaling reflecting the size of the ratio $L/ \sqrt{t/\log t}$ when $L \sim \sqrt{t/\log t}$ as $t \to \infty$.
Moreover, when $L \to \infty$ such that $L = o(\sqrt{t/\log t})$, we find different $L$-dependent scalings under which
another Gumbel law is obtained for \eqref{gen-scaling} in the limit, namely $a_t = a_t^{(L)}$ and $b_t=b_t^{(L)}$ for 
\begin{equation}\label{eq:Lscaling}
	a^{(L)}_t = \log \left( \frac{L^2}{\sqrt{2\pi \log L^2}}\right) \qquad\mbox{and}\qquad b^{(L)}_t = \sqrt{\frac{t}{\log L^2}}. 
\end{equation}
In this case, the displacement of the leading particle is of order $\sqrt{t \log L}$ rather than $\sqrt{t\log t}$ as in the first two limits. 

The transition at $b_t=\sqrt{t/\log t}$, which may be interpreted as the order of the standard deviation of $X_t$ in the limit \eqref{eq:RWGumbel}, indicates that the size $L$ of the block of particles needed for the rightmost particle to behave as if it were trailed by infinitely many particles. Equivalently, it is essentially only the particles within $O(\sqrt{t/\log t})$ distance behind the leader that have significant influence on its behavior. 
This result is presented in Theorem \ref{mainLthm}.

\subsection{Proof technique}\label{prooftech}

Our analysis of $X_t$ in Theorems \ref{main} and \ref{mainLthm} is based on recently established negative association properties for the exclusion process on $\Z$. 
In particular, we look at the limiting behavior of the process
\begin{equation}\label{eq:N}
	N_t = \sum_{k > z} \eta_t(k) 
\end{equation}
where $z = \sigma b_t(x+a_t)$ for $x\in \R$ and $\sigma^2 = \sum_i i^2 p_i$.
In words, $N_t=N_t(x)$ counts the number of particles that have moved to the right beyond $z=z(t,x)>0$ at time $t$. $N_t$ has a convenient relationship to the position of the leading particle in that $X_t \le z$ if and only if no particles lie to the right of $z$. That is, 
\[
	\left\{ \frac{X_t}{\sigma b_t} - a_t \le x \right\} = \{X_t \le z\} = \{N_t = 0\}. 
\]
Such a relation was used in \cite{Arr83} to show the law of large numbers \eqref{Arratia-LLN} and tightness of $(\sigma b_t)^{-1}X_t-a_t$ (in the nearest-neighbor, full step setting) by bounding $\sup_tP(N_t\ge1) \le \sup_tE[N_t]<\infty$.

To identify the limiting distribution of the scaled leading-particle process, it would be enough to determine the limiting distribution of $N_t$. In this sense, our contribution is to determine appropriate scalings $a_t$ and $b_t$ in the settings of Theorems \ref{main} and \ref{mainLthm} and show that $N_t$ converges as $t \to \infty$ to a Poisson random variable with a parameter $\lambda_x$, depending on the initial condition and jump probabilities.
Then, $N_t \Rightarrow \mbox{Poisson}(\lambda_x)$ implies 
\[
	P\left( \frac{X_t}{\sigma b_t} - a_t \le x \right) \to P\left( \mbox{Poisson}(\lambda_x) = 0\right) = e^{-\lambda_x}. 
\]

As we show in Theorem \ref{ArrMean}, $E[N_t] \to \sigma e^{-x}$ as $t\to\infty$ in the full-step setting, which then gives the Gumbel cumulative distribution function $e^{-\lambda_x} = e^{-\sigma e^{-x}}$. In this case and in others, computing the limit of $E[N_t]$ for initial measure $\nu(\eta(k) = 1) = \rho_k$ is aided by the representation 
\begin{equation}\label{eq:firststir}
	E[N_t] = \sum_{i \le 0} \rho_iP(\xi_i(t) > z), 
\end{equation}
where each $\xi_i$ is a random walk based on $\{p_i\}$ with $\xi_i(0)=i$. This comes from the ``stirring'' construction of the process $\eta_t$, which we discuss in more detail in Subsection \ref{stirring}. We note here that Condition \ref{ic} implies that $\sum_{i\le0}\rho_i = \infty$, i.e., there are an infinite number of particles in the system. With only a finite number, $\sum_{i\le0}\rho_i < \infty$, in which case the mean in \eqref{eq:firststir} goes to zero as $t\to\infty$ with $z = \sigma b_t(x+a_t)$ for the scaling in \eqref{eq:ArrScaling} or \eqref{eq:Lscaling}, and there would be no nontrivial Poisson limit for $N_t$. 

In a system of independent particles, $N_t$ is the sum of independent Bernoulli random variables, and so showing a Poisson limit essentially comes down to computing the limit of $E[N_t]$, which can be done for periodic initial step conditions, as is done in \cite{Arr83} for the ``full step'' initial condition. 
However, in the exclusion system, $N_t = \sum_{i\leq 0}\eta_0(i)1_{\{\xi_i(t)>z\}}$ is a sum involving correlated stirring variables.  Then, to deduce Poisson limits,
we also must show that
the correlations between each pair of particle positions vanish in the limit.

 This is a challenging problem, and occupies a significant part of our analysis. We use a relatively recently developed theory of ``strong Rayleigh'' and negative association properties for the symmetric exclusion process \cite{BorBraLig09,Lig09,Van10}, which allows for Poisson limit theorems despite the dependence induced by exclusion interaction 
  (see
Subsection \ref{neg_dep_subsect}). Also, duality properties for symmetric exclusion and semigroup inequalities between symmetric exclusion and the system of independent random walks are used to rewrite the sum of correlations between particle positions in an analyzable form (see Subsection \ref{duality}).   
  
  In particular, after this formulation, the proof consists of involved combinations of sharp random walk local limit theorems and tail estimates, facilitated by Condition \ref{mgf}, to deduce the result (see Section \ref{covcomp}).
  In the scheme of the estimates, we need take account of the different scalings in our main theorems as well as the shapes of the initial distributions, whether a full step or $L$-step, to deduce the desired correlation bounds. 
These arguments are given for $\{p_i\}$ with finite moment generating function, however we provide improved rates of convergence when $\{p_i\}$ is finite range.

\subsection{A remark on ASEP}\label{comASEP}

Fewer results are known for the tagged particle in the asymmetric simple exclusion process (ASEP). 
With respect to nearest-neighbor interactions with $p = p(x,x+1) = 1 - p(x, x-1) = 1 - q$ and $p \ne q$,
one can consider the limiting behavior of the tagged particle, initially at the origin, when the system starts from a full-step initial profile.  The generator for the process, acting on functions that depend on a finite number of occupation variables, is given by
$$\sL_{ASEP}f(\eta) = \sum_{x\in \Z} \left\{p\left(f(\eta^{x,x+1})-f(\eta)\right) + q\left(f(\eta^{x,x-1})-f(\eta)\right)\right\}.$$
For the case of $p > q$, the tagged particle was studied 
in
\cite{TraWid09}:  interestingly,
 under the expected diffusive scaling, the tagged particle position converges to a non-Gaussian limit given in terms of a determinantal formula. Related limits are also stated in 
\cite{TraWid11} for other types of step profiles, such as alternating ones.

Although we focus on the symmetric system in this article, we may make a remark to complete the picture for the nearest-neighbor, full step situation by
considering $p < q$ and $\rho_x = 1$ for $x \le 0$ and $\rho_x=0$ otherwise.  Here, since $p<q$, the tagged particle does not wander far from the ``bulk''.
Spacings between particles in the exclusion process with nearest-neighbor interaction can be mapped to a zero range process $\{\zeta_t\}$ on $\N$ with an infinite well of particles at $x=0$. 
The position of the tagged particle initially at the origin can then be expressed as $X_t = \sum_{x\ge1}\zeta_t(x)$.
One expects then that the law of $X_t$ should
converge weakly, without scaling as $p<q$, to $\mu(\sum_{x\ge 1}\zeta(x) \in \cdot)$ in terms of an invariant measure $\mu$ for the zero range process. This is formulated in Proposition \ref{ASEP}.

We comment in passing that, in contrast, when starting under stationarity, it was shown in \cite{Kip86} that for the nearest-neighbor process and under diffusive scaling, the centered position $X_t$ is asymptotically normal when the system is started from a stationary Bernoulli$(\rho)$ product measure. In the totally asymmetric case of $p = 1$ and when $\rho < 1$, this result is a consequence of Burke's theorem from queueing theory.
The result in \cite{Kip86} was improved in \cite{FerFon96} and \cite{Goncalves}, wherein it was shown that the properly centered and scaled process $X_t$ converges as a process to Brownian motion. 

Moreover, other central limit theorems of type \eqref{eq:kipvar} under stationary initial condition were proved in \cite{SetVarYau00} and \cite{Var95} with respect to $d\geq 3$ and mean-zero non-nearest neighbor asymmetric systems, respectively.
We also refer to \cite{Set06} for a tagged particle variance calculation showing diffusivity in low dimensions.  
See also \cite{Rezakhanlou,Saada} for a law of large numbers for the tagged particle in asymmetric exclusion starting from initial conditions with a macroscopic density of particles around the origin and \cite{Grig,Wang1} for results on variable-speed or driven tagged particles in asymmetric systems.  In \cite{Ferrari}, a review of some of these and other results are presented.

\subsection*{Notation}
We now list some notation and conventions we will follow in the article. For a probability measure $\nu$ on $\{0,1\}^{\Z}$, $\P_\nu$ denotes the probability measure on the sample space under which $\eta_0$ is distributed according to $\nu$, and $\E_\nu$ denotes the corresponding expectation. Throughout, $\{\xi_t\}$ will denote a continuous-time random walk with jump distribution $p_i$. For $y \in \Z$, $P_y$ denotes the probability measure under which $\xi_0 = y$. Unless otherwise stated, $X$ denotes a standard normal random variable. For sequences $x_t$ and $y_t$, $x_t \sim y_t$ as $t \to\infty$ means $\lim_{t\to\infty} x_t/y_t = 1$, $x_t = O(y_t)$ as $t \to \infty$ means $x_t \le C y_t$ for some positive constant $C$ and sufficiently large $t$, and $x_t = o(y_t)$ as $t \to \infty$ means $\lim_{t\to\infty} x_t/y_t = 0$. When it is clear from context, we will suppress $t \to \infty$ and simply write $x_t \sim y_t$, $x_t = O(y_t)$, and $x_t=o(y_t)$.

\subsection*{Contents}
We state and remark on the main results, Theorems \ref{main}, \ref{mainLthm}, and Proposition \ref{ASEP}, in Section \ref{mainresults}.  After preliminaries on the stirring process representation, negative association and Poisson limits, and duality properties of symmetric exclusion processes in Section \ref{propsofex}, we give an outline of the proofs of Theorems \ref{main} and \ref{mainLthm} in Section \ref{proofs}, which makes use of results on mean limits and covariance bounds in Sections \ref{meancomp} and \ref{covcomp}.  Finally, collected in the Appendix are the random walk probability estimates used.

\section{Main results}\label{mainresults}

Here we present our main results, which are proved in Section \ref{proofs}. For all of what follows, we assume the following on the jump distribution. 
\begin{condition}\label{mgf} For some $\theta > 0$, $\sum_{i\in \Z} e^{\theta i}p_i < \infty$.
\end{condition}
Recall that $\sigma^2 = \sum_i i^2 p_i$. 
The following condition defines the class of periodic step initial profiles we consider. 
\begin{condition}\label{ic} The initial measure $\nu$ on $\{0,1\}^\Z$ satisfies $\nu(\eta(x) = 1) = \rho_x \in [0,1]$ for each $x \in \Z$, where 
\begin{enumerate}
\item[(a)] $\rho_0 = 1$ and $\rho_x = 0$ for $x > 0$, and 
\item[(b)] For some integer $m \ge 1$, $\rho_i = \rho_{i-m}$ for each $i \le -1$ and $\rho_{-1}, \rho_{-2}, \ldots, \rho_{-m}$ are not all $0$. 
\end{enumerate}
\end{condition}

For an initial condition satisfying Condition \ref{ic}, let 
\begin{equation}\label{eq:rhobar}
	\bar\rho = \frac{1}{m}\sum_{j=1}^m \rho_{-j}. 
\end{equation}
Our first result is the following, which solves the open problem in \cite{Arr83} as a special case when $\rho_i = 1_{\{i\leq 0\}}$ and $p_{-1}=p_1=1/2$. 

\begin{theorem}\label{main} Suppose that $\nu$ satisfies Condition \ref{ic}, and let $\bar\rho$ be as in \eqref{eq:rhobar}. 
Then for all $x\in \R$, 
\[
	\lim_{t\to\infty}\P_\nu\left( \frac{X_t}{\sigma b_t} - a_t \le x\right) = \exp\left( - \sigma \bar\rho e^{-x} \right) , 
\]
where $a_t$ and $b_t$ are the scalings in \eqref{eq:ArrScaling}. 
\end{theorem}

Our next result probes the sensitivity of the limit in Theorem \ref{main} to the number of particles behind the lead particle, as discussed in the introduction, with respect to a sequence of processes with the ``$L$-step'' initial conditions. 

\begin{theorem}\label{mainLthm} 
Suppose that $\nu$ satisfies Condition \ref{ic} and let $\bar \rho$ be as in \eqref{eq:rhobar}. For $L=L(t) > 0$, define the measure $\nu_L$ on $\{0,1\}^\Z$ by 
\[
	\nu_L(\eta(k) = 1) = \left\{\begin{array}{cl} \rho_k &\mbox{if}\;\; -L \le k \le 0, \\ 0 &\mbox{otherwise}, \end{array}\right. 
\]
and suppose that $L \to \infty$ as $t \to \infty$. 
\begin{enumerate}
\item[(a)] If $L \sqrt{\frac{\log t}{t}} \to c \in (0, \infty]$ and $a_t$ and $b_t$ are as in \eqref{eq:ArrScaling}, then 
\[	
	\lim_{t\to\infty}\P_{\nu_L}\left( \frac{X_t}{\sigma b_t} - a_t \le x \right) = \exp\left(-\sigma\bar\rho(1 - e^{-c/\sigma})e^{-x}\right). 
\]

\item[(b)] If $L \sqrt{\frac{\log t}{t}} \to 0$ and $a_t=a^{(L)}_t$ and $b_t=b_t^{(L)}$ are as in \eqref{eq:Lscaling}, then 
\[	
	\lim_{t\to\infty}\P_{\nu_L}\left( \frac{X_t}{\sigma b_t} - a_t \le x \right) = \exp\left(-\bar\rho e^{-x} \right). 
\]
\end{enumerate}
\end{theorem}

\begin{remark} \upshape In the context of Theorem \ref{mainLthm}(a), the result when $c=\infty$ matches that of Theorem \ref{main}. When $0<c<\infty$, the mean of the Gumbel is less than when $c=\infty$, a reflection of the smaller size of the $L$-step profile. However, sending $c \to 0$ does not recover the result of Theorem \ref{mainLthm}(b). We attribute this to the difference in scalings between parts (a) and (b): in the scaling of part (a) with $L = o(t^{1/2}(\log t)^{-1/2})$, we would obtain the trivial limit of $1$. It is also interesting to note that $\sigma$ does not appear in the limit of Theorem \ref{mainLthm}(b) with respect to the more dilute
system when $L = o(t^{1/2}(\log t)^{-1/2})$. 
\end{remark}

\begin{remark}\label{mth} \upshape
We note that our proof techniques may be straightforwardly extended to determine the limiting laws of the order statistics of the processes $\eta_t$ as follows. If, for $m\ge0$, $X_t^{(m)}$ denotes the position of the $m$th right-most particle at time $t$, beginning with $X^{(0)}_t = X_t$, then $X_t^{(m)} \le z$ if and only if $N_t \le m$, where $N_t$ is given in \eqref{eq:N}. In the case of nearest-neighbor interactions, 
a particle cannot make a jump over another, and so this gives the limiting distribution for the position of the tagged particle initially $m$th from the right.
In the proofs of Theorem \ref{main} and Theorem \ref{mainLthm} (a) and (b), we show that $N_t$ converges weakly to a Poisson distribution with means $\sigma\bar\rho e^{-x}$, $\sigma\bar\rho(1 - e^{-c/\sigma})e^{-x}$, and $\bar\rho e^{-x}$, respectively. 

Therefore, in the setting of Theorem \ref{main}, as $t \to \infty$,
\begin{equation}\label{eq:orderstat}
	\P_\nu\left( \frac{X_t^{(m)}}{\sigma b_t} - a_t \le x \right) \to 
	\sum_{k=0}^m \frac{(\sigma\bar\rho)^k e^{-\sigma\bar\rho e^{-x} - kx}}{k!}. 
\end{equation}
 The analog in the setting of Theorem \ref{mainLthm}(a) is 
\begin{equation}\label{eq:Lorderstat1}
	\P_{\nu_L}\left( \frac{X^{(m)}_t}{\sigma b_t} - a_t \le x \right) \to \sum_{k=0}^m \frac{[\sigma\bar\rho(1-e^{-c/\sigma})]^k}{k!} \exp\left(-\sigma\bar\rho(1 - e^{-c/\sigma})e^{-x} - kx\right), 
\end{equation}
and in the setting of Theorem \ref{mainLthm}(b) is 
\begin{equation}\label{eq:Lorderstat2}
	\P_{\nu_L}\left( \frac{X^{(m)}_t}{\sigma b_t} - a_t \le x \right) \to \sum_{k=0}^m \frac{\bar\rho^k e^{-\bar\rho e^{-x} - kx}}{k!}. 
\end{equation}

One may compute from \eqref{eq:orderstat}, \eqref{eq:Lorderstat1}, and \eqref{eq:Lorderstat2} that the means of the limit distributions of the scaled $m$th and $(m+1)$th particle positions differ by $(m+1)^{-1}$ (independent of the values of $\sigma$, $\bar\rho$, or $c$). Then for large $t$, the difference $X_t^{(m)} - X_t^{(m+1)}$ is roughly of order $b_t/(m+1)$, and the gaps between particles eventually diverge. 
 This further aids the intuition as to why the exclusion behavior matches that for independent particles, since a particle in the exclusion system moves as an unconstrained random walk in the absence of others nearby. 
\end{remark}

\begin{remark}\label{robust} \upshape As will be seen in the proofs of Theorems \ref{main} and \ref{mainLthm}, 
our results are robust in that a Gumbel limit can be obtained as long as $\E_\nu[N_t]$ converges. So, the results still hold if, say, a finite number of $\rho_i$ do not satisfy the periodicity condition (as $P(\xi_i(t)>z)$ vanishes in \eqref{eq:firststir} as $t\rightarrow\infty$ for each $i\leq 0$). In theory this means that initial conditions other than the periodic type in Condition \ref{ic} can lead to similar limiting distributions.  However in practice computing $\lim_{t \to\infty}\E_\nu[N_t]$ for arbitrary $\nu$ is difficult. Theorems \ref{main} and \ref{mainLthm} could alternatively be stated with the assumption that $\lim_{t\to\infty}\E_\nu[N_t] = \lambda_x \in (0,\infty)$, and the limiting distributions would then be characterized in terms of this limit $\lambda_x$. In this paper, we focus on general periodic initial conditions where concrete evaluations can be done. 
In particular, these initial profiles admit a comparison with the full-step $\rho_k \equiv 1$ case, for which the limit of $\E_\nu[N_t]$ can be obtained generally (See Theorems \ref{ArrMean} and \ref{Lthm}). 
\end{remark}

Lastly, we remark on the behavior of nearest-neighbor ASEP with drift in the direction of the step. As noted in Subsection \ref{comASEP}, spacings $\zeta_t(x)$, between the particles beginning at $-(x+1)$ and $-x$, in the process $\{\eta_t\}$ correspond to a zero range process $\{\zeta_t\}$ on $\Omega = \big\{\eta\in \N^{\N_+}: \sum_{x\geq 1}\eta(x)<\infty\big\}$ understood with an infinite well of particles at $x = 0$. Here, $\N = \{0, 1, 2, \ldots\}$.  
The generator 
for this process is given by its action on local functions as 
\begin{align*}
	\sL_{\text{\tiny ZR}}f(\zeta) 
	&=  p \left[ f(\zeta^{0,1}) - f(\zeta) \right] 
	+ \sum_{x\ge1}1_{\{\zeta(x) \ge 1\}} \Big\{  p\left[f(\zeta^{x,x+1}) - f(\zeta) \right]  + q\left[ f(\zeta^{x,x-1}) - f(\zeta) \right] \Big\},
\end{align*}
where 
\[
	\zeta^{x,y}(z) = \left\{ \begin{array}{cl} \zeta(x) - 1 &\;\mbox{if}\; z = x, \\
								\zeta(y) + 1 &\;\mbox{if}\; z = y, \\
								\zeta(z) &\;\mbox{if}\; z \ne x,y, \end{array}\right. 
\]
for $x, y \ge 1$ and 
\[
	\zeta^{0,1}(x) = \left\{ \begin{array}{cl} \zeta(x) + 1 &\;\mbox{if}\; x = 1, \\
							         \zeta(x) & \;\mbox{if}\; x > 1, \end{array}\right. 
	\qquad
	\zeta^{1,0}(x) = \left\{ \begin{array}{cl} \zeta(x) - 1 &\;\mbox{if}\; x = 1, \\
							         \zeta(x) & \;\mbox{if}\; x > 1. \end{array}\right. 					      
\]
The full step initial profile corresponds to $\zeta \equiv 0$ in the zero range context.  The displacement of the tagged particle may be expressed as 
\[
	X_t = \sum_{x \ge 1} \zeta_t(x). 
\]
Moreover, the $m$th right-most particle position $X^{(m)}_t$, that is the one starting at $-m$ for $m\geq 0$, satisfies
$$ X^{(m)}_t + m = \sum_{x \geq m+1}\zeta_t(x).$$ 
However, in what follows, to be brief, we concentrate on the behavior of $X_t = X^{(0)}_t$.

Let $\mu$ be the product measure on $\N^{\N_+}$ with Geometric$(1 - (p/q)^x)$ marginals. That is, 
\[
	\mu\left( \eta : \eta(x) = k \right) = \left( p/q\right)^{kx} \left( 1 - \left( p/q\right)^x \right), \qquad k \ge 0, 
\]
for each $x \in \N_+$.  One may calculate by the relation $E_\mu[\sL_{\text{\tiny ZR}}f(\zeta) ]=0$ that $\mu$ is an invariant measure, which is unique by irreducibility of this zero range process on $\Omega$.

\begin{prop}\label{ASEP} If $\zeta(x) = 0$ for all $x \ge 1$ and $\P_{\zeta}$ denotes the probability measure under which $\zeta_0 = \zeta$, then 
\begin{equation}\label{eq:ASEPlimit}
	\P_{\zeta}(X_t \in \cdot) \to \mu \left( \sum_{x\ge1} \zeta(x) \in \cdot \right), \qquad t \to \infty.
\end{equation}
\end{prop}

\begin{proof} As the rate $g(k)=1_{\{k\geq 1\}}$ is increasing in $k$, the zero range system is ``attractive''; see \cite{Andjel}. Hence, we may construct a ``basic coupling'' $(\zeta_t, \xi_t)$ where $\xi_0 \sim \mu$, so that initially $\xi_0(x) \ge \zeta_0(x) = 0$ for all $x \ge 1$ and, for all times $t>0$, $\xi_t(x) \ge \zeta_t(x)$.
Indeed, the coupled process has generator
\begin{align*}
	&\overline{\sL}_{\text{\tiny ZR}} f(\zeta,\xi) 
	= p\left[f(\zeta^{0,1},\xi^{0,1}) - f(\zeta,\xi)\right] \\
	&\qquad\qquad\quad+ \sum_{x\ge1} \bigg\{ 1_{\{\zeta(x)\wedge\xi(x) \ge 1\}}\Big( p\left[f(\zeta^{x,x+1},\xi^{x,x+1}) - f(\zeta,\xi)\right] + q\left[f(\zeta^{x,x-1},\xi^{x,x-1})-f(\zeta,\xi) \right] \Big)\\
	&\qquad\qquad\qquad\qquad\;\;+  1_{\{\zeta(x) \ge 1 > \xi(x)=0\}}\Big( p\left[f(\zeta^{x,x+1},\xi) - f(\zeta,\xi)\right] + q\left[f(\zeta^{x,x-1},\xi)-f(\zeta,\xi) \right] \Big)\\
	&\qquad\qquad\qquad\qquad\;\;+ 1_{\{0=\zeta(x) < 1 \le \xi(x)\}}\Big( p\left[f(\zeta,\xi^{x,x+1}) - f(\zeta,\xi)\right] + q\left[f(\zeta,\xi^{x,x-1})-f(\zeta,\xi) \right] \Big) \bigg\}.
\end{align*}
Now let $Y_t = \sum_{x \ge 1} \xi_t(x)$, so that $X_t \le Y_t$ for all times $t \ge 0$. Since $\xi_t$ begins at stationarity, $\P_{\mu}(Y_t \in \cdot)$ is equal to the right hand side of \eqref{eq:ASEPlimit} at all times $t$. 
Thus, if $f$ is a Lipschitz-$1$ function and $E_\mu$ denotes expectation on $\N^{\N_+}$ with respect to $\mu$, 
\[
	\left| \E_{\zeta} \left[ f(X_t) \right] - E_\mu \left[ f\left( \sum_{x\ge1} \zeta(x) \right) \right] \right| 	\le \E_{\zeta,\mu} \left[ Y_t - X_t \right]. 
\]
To establish the desired weak convergence, by a Portmanteau theorem it suffices to show that $\E_{\zeta,\mu} \left[ Y_t - X_t \right] \to 0$ as $t \to \infty$. Since $\zeta_t \Rightarrow \mu$, we have $\zeta_t(x) \Rightarrow \mu(\zeta(x) \in \cdot)$ for each $x \ge 1$. 
Moreover, as
\[
	\sup_{t\ge0} \E_{\zeta}\left[ \zeta_t(x) 1_{\{\zeta_t(x) > M \}} \right] \le \sup_{t\ge0} \E_{\mu}\left[ \xi_t(x) 1_{\{\xi_t(x) > M \}} \right]  = E_\mu \left[ \zeta(x)1_{\{\zeta(x) > M\}} \right] \to 0
\]
as $M \to \infty$, the collection $\{\zeta_t(x) : t \ge0\}$ is uniformly integrable, and
$\E_{\zeta} \left[ \zeta_t(x) \right] \to E_\mu \left[ \zeta(x) \right]$
for each $x$. It follows that 
$\E_{\zeta,\mu} \left[ \xi_t(x) - \zeta_t(x) \right] \to E_\mu \left[ \zeta(x) \right]  - E_\mu \left[ \zeta(x) \right] = 0$
as $t \to \infty$. Since $0\leq \xi_t - \zeta_t \le \xi_t$ for each $t$ and $\E_{\mu} \left[ \sum_{x\ge1} \xi_t(x) \right] = E_\mu \left[ \sum_{x\ge1} \zeta(x)\right] < \infty$, it follows by the dominated convergence theorem that 
\[	
	\E_{\zeta,\mu} \left[ Y_t - X_t \right] = \sum_{x\ge1} \E_{\zeta,\mu} \left[\xi_t(x) - \zeta_t(x)\right] \to 0 \qquad \mbox{as}\qquad t\to\infty.\qedhere
\] 
\end{proof}

\section{Properties of the symmetric exclusion process}\label{propsofex}

In the following sections, we collect some properties of the symmetric exclusion process that motivate our main results and which will be useful in their proofs. 

\subsection{The stirring process}\label{stirring}

The symmetric exclusion system starting from a step profile can be expressed as a system of random stirrings 
\[
	\left\{ \xi_i(t) : i \in \Z_-, t \ge 0 \right\}, 
\]
where at most one of the $\{\xi_i(t)\}$ can occupy a given site at one time, and each pair of particles at $x < y$ are interchanged after an exponential amount of time with rate $p_{y-x}$, with each pair of locations having an independent clock. Equivalently, for each $i$, $\{\xi_i(t) : t\ge0\}$ is a random walk on $\Z$ based on $\{p_j\}$ starting at $i$, and for each $t$, $\{\xi_i(t) : i\in \Z_-\}$ is a random permutation of $\Z$ where $\{\xi_i(0)\}$ is the identity permutation and which appends the transpositions $(k, k+j)$ or $(k, k-j)$ after an exponential amount of time with rate $p_j$ independently. The stirring representation provides one way of constructing the process $\eta_t$: Namely, $\eta_t(j) = \sum_{i\in \Z} \eta_0(i)1_{\{\xi_i(t) = j\}}$.
For further details on its construction and basic properties, we refer to \cite[Ch. VIII]{Lig85}.

We can express the random variable $N_t$ in \eqref{eq:N} in terms of the stirring variables and initial configuration $\eta_0$ as 
\begin{equation}\label{eq:stirringN}
	N_t = \sum_{i \le 0} \eta_0(i)1_{\{\xi_i(t) > z\}}, 
\end{equation}
which is useful for computing the limit of its mean. In particular, if $\nu$ is a Bernoulli initial condition as in Condition \ref{ic}, then since each $\xi_i$ is marginally a random walk beginning at $i$, 
\[
	\E_\nu[N_t] = \sum_{i \le 0} \rho_i P_i(\xi_t > z). 
\]

Note, for later use, by translation-invariance and symmetry of the random walk $\xi_t \stackrel{d}{=} - \xi_t$ that
$P_i(\xi_t >w) = P_0(\xi_t> w-i) = P_0(\xi_t< -w +i)$.

\subsection{Negative dependence}
\label{neg_dep_subsect}

The symmetric exclusion process obeys the following correlation inequality due to Andjel \cite{And88}: for $A \subset \Z$ finite, 
\begin{equation}\label{eq:NegCor}
	P(\eta_t \equiv 1 \mbox{ on } A) \le \prod_{x \in A} P(\eta_t(x)=1). 
\end{equation}
That is, the particles tend to spread out more than they would if instead they moved independently of each other. In fact, the values $\{\eta_t(k): k \in \Z\}$ are {\em strong Rayleigh} for each $t > 0$ when $\eta_0$ is distributed according to a product measure (as is the case for the initial step distributions we consider in Condition \ref{ic}). 

That is, for any finite subset $A \subset \Z$, the generating function $Q(x) = \E_{\nu}\left[ \prod_{i\in A} x_i^{\eta_t(i)}\right]$ satisfies 
\[
	\partial_{x_i} Q(x)\partial_{x_j}Q(x) \ge Q(x)\partial^2_{x_i,x_j}Q(x)
\]
for all $i \ne j$ and all $x \in \R^{A}$. (Note that plugging $x_k = 1$ for all $k$ into the above display recovers negative correlation, namely \eqref{eq:NegCor} for $A$ replaced with $\{i,j\}$.) For further details we refer to \cite{BorBraLig09} (see also \cite{Pem00}). 

The strong Rayleigh property implies that for each $t > 0$ and $B \subset \Z$, there exist independent Bernoulli random variables $\{\zeta_{t}(k),  k \in B\}$ such that 
\begin{equation}\label{eq:zetasum}
	\sum_{k\in B} \eta_t(k) \overset{d}{=} \sum_{k \in B} \zeta_{t}(k).
\end{equation}
For the proof of \eqref{eq:zetasum} for finite $B$, see \cite[Proposition 4]{Lig09}, and for an extention to arbitrary $B\subset\Z$, see Proposition 1 and the discussion in Section 4 of \cite{Van10}.
In \cite{Lig09, Van10}, given \eqref{eq:zetasum}, central limit theorem for sums in the exclusion system were stated, and the following Poisson limit theorem was shown (see Proposition 5 in \cite{Lig09}).
 For further details of the properties above, we refer the reader to the discussions in those works. 

\begin{lemma}[Liggett, Vandenberg-Rodes]\label{liggettpoisson} If, as $t\to\infty$, 
\begin{enumerate}[(i)]
\item $\sum_{k} \E_\nu\left[\eta_t(k)\right] \to \lambda$, 
\item $\sum_{k} \left(\E_\nu[\eta_t(k)]\right)^2 \to 0$, and 
\item $\sum_{j\ne k} \Cov_\nu(\eta_t(j), \eta_t(k)) \to 0$, 
\end{enumerate}
then, 
\[
	\sum_k \eta_t(k) \Rightarrow \mbox{Poisson}\,(\lambda). 
\]
\end{lemma}

We will use this result to show a Poisson limit for the process $N_t$, which will follow from verifying (i)--(iii). We note that the negative correlation property \eqref{eq:NegCor} implies that $\Cov_\nu(\eta_t(j), \eta_t(k)) \le 0$ for all $j \ne k$. Hence we verify (iii) by finding an upper bound for 
the negative sum of covariances. For convenience, we introduce the notation 
\begin{equation}\label{eq:covsum}
	\sC_t(\nu,z) = - \underset{j \ne k}{\sum_{j,k > z}} \Cov_\nu(\eta_t(j), \eta_t(k)). 
\end{equation}
The stirring variables from the representation presented in Section \ref{stirring} have their own negative dependence property, as shown in Lemma 1' of \cite{Arr83}:

\begin{lemma}[Arratia]\label{stirringdep} For symmetric $\{p_i\}$, $A \subset \Z$, and $i \ne j$, the events $\{\xi_i(t) \in A\}$ and $\{\xi_j(t) \in A\}$ are negatively correlated. 
\end{lemma}

The previous lemma allows for the following result, which is used in the proofs of Theorems \ref{main} and \ref{mainLthm} and may be skipped on first reading. 

\begin{lemma}\label{fullstepdom} Let $\eta \in \{0,1\}^\Z$ be defined by $\eta(k) = 1$ for $k \le 0$ and $\eta(k)=0$ otherwise. For any Bernoulli product measure $\nu$ on $\{0,1\}^\Z$ with $\nu(\eta(k) = 1) = 0$ for $x > 0$, and for any $z \in \R$, 
\begin{equation}\label{eq:onetoprove}
	\sum_{k > z} \left(\E_\nu[\eta_t(k)]\right)^2 + \sC_t(\nu,z) \le \sum_{k > z} \left(\E_{\eta}[\eta_t(k)]\right)^2 + \sC_t(\eta,z). 
\end{equation}
For each $L > 0$, define $\eta_L \in \{0,1\}^\Z$ by $\eta_L(k) = 1$ for $-L \le k \le 0$ and $\eta_L(k) = 0$ otherwise. Then for $\nu_L$ as in Theorem \ref{mainLthm}, 
\[
	\sum_{k > z} \left(\E_{\nu_L}[\eta_t(k)]\right)^2 + \sC_t(\nu_L,z) \le \sum_{k > z} \left(\E_{\eta_L}[\eta_t(k)]\right)^2 + \sC_t(\eta_L,z). 
\]
\end{lemma}

\begin{proof} We prove \eqref{eq:onetoprove}; the proof of the second inequality is the same. Recall that for each $k \le 0$, $\rho_k = \nu(\eta(k) = 1)$. Let $I_i = 1_{\{\xi_i(t) > z\}}$ for each $i \le 0$. 
From the representation \eqref{eq:stirringN}, we calculate
\begin{align*}
	\E_{\eta}[N_t] - \Var_{\eta}(N_t) &= \sum_{i\le 0} (E[I_i])^2 - \sum_{i\ne j} \Cov(I_i,I_j), \ \ {\rm and}\\
	\E_\nu[N_t] - \Var_\nu(N_t) &= \sum_{i\le0}\rho_i^2(E[I_i])^2 - \sum_{i\ne j}\rho_i\rho_j\Cov(I_i,I_j). 
\end{align*}
By Lemma \ref{stirringdep}, $I_i$ and $I_j$ are negatively correlated for $i\ne j$. Hence, $-\Cov(I_i,I_j) > 0$ and then 
\[
	0 \le \E_\nu[N_t] - \Var_\nu(N_t) \le \E_{\eta}[N_t] - \Var_{\eta}(N_t). 
\]
On the other hand, we may use the representation $N_t = \sum_{k>z}\eta_t(k)$ to compute 
\[
	\E_\mu[N_t] - \Var_\mu(N_t) = \sum_{k>z} (\E_\mu[\eta_t(k)])^2 + \sC_t(\mu,z), 
\]
for $\mu = \nu, \delta_\eta$. The result follows. 
\end{proof}

\subsection{Duality}\label{duality}
 
As we will see in Lemma \ref{fullstepdom}, we are able to limit much of our analysis to deterministic initial conditions. Let $\eta \in \{0,1\}^\Z$. 
As mentioned in the preceeding discussion, negative dependence in the symmetric exclusion process implies that for all $j \ne k$, 
\begin{equation}\label{eq:NegCov}
	-\Cov_\eta(\eta_t(j), \eta_t(k)) = \E_\eta[\eta_t(j)] \E_\eta[\eta_t(k)] - \E_\eta[\eta_t(j)\eta_t(k)] \ge 0. 
\end{equation}
This covariance can be computed using the self-duality of the symmetric exclusion process, which can be expressed as follows. If $(\zeta_1(t),\ldots, \zeta_n(t))$ denotes the particle positions at time $t$ of an $n$-particle system based on $\{p_i\}$ with exclusion interaction, then 
\[
	\E_\eta[\eta_t(x_1) \cdots \eta_t(x_n)] = E_{(x_1, \ldots, x_n)}\left[\eta(\zeta_1(t))\cdots \eta(\zeta_n(t))\right], 
\]
where $\{x_1, \ldots, x_n\} \subset \Z$ and $E_{(x_1, \ldots, x_n)}$ denotes expectation with respect to which $\zeta_i(0)=x_i$ \cite[Ch. VIII, Thm. 1.1]{Lig85}. In particular, noting \eqref{eq:NegCov}, we may consider $n = 2$ and write, for $j\neq k$,
\[
	-\Cov_\eta(\eta_t(j), \eta_t(k)) = \left[U_2(t) - V_2(t)\right]\eta(j)\eta(k), 
\]
where $\{V_2(t), t\ge0\}$ denotes the semigroup of the process $(\zeta_1(t), \zeta_2(t))$ and $\{U_2(t), t\geq 0\}$ is the semigroup of a pair of independent random walks with the common law of $\{\xi_t, t\geq 0\}$.  Here, $h(j,k)=\eta(j)\eta(k)$ refers to the map $(j,k)\mapsto \eta(j)\eta(k)$ in terms of the specified configuration $\eta$, so that 
$$U_2(t)\eta(j)\eta(k) =E_j[\eta(\xi_t)] E_k[\eta(\xi_t)]=\E_\eta[\eta_t(j)] \E_\eta[\eta_t(k)].$$
When $j=k$, we will write $h(j,j) = \eta(j)\eta(j)$ and will understand $U_2(t)h(j,j) = \left(E_j[\eta(\xi_t)]\right)^2$ in the following.

Note that $V_2$ is a symmetric operator on the space of real functions on $\{(j,k)\in \Z^2: j\neq k\}$.  

 The semigroups $V_2$ and $U_2$ have corresponding generators:
\begin{equation}\label{eq:ExPairgen}
\begin{aligned}
	\sV f(x,y) &= \sum_{z \ne y} p_{z-x}[f(z,y)-f(x,y)] + \sum_{z \ne x} p_{z-y}[f(x,z)-f(x,y)], 
\end{aligned}
\end{equation}
and 
\begin{equation}\label{eq:RWgen}
\begin{aligned}
	\sU f(x,y) &= \sum_{z \in \Z} \left( p_{z-x}[f(z,y)-f(x,y)] + p_{z-y}[f(x,z)-f(x,y)]\right).
\end{aligned}
\end{equation}

Comparing the independent and exclusion two-particle systems is useful for further computation. In fact, 
\begin{equation}\label{eq:VUcomp}
	V_2(t)g(j,k) \le U_2(t)g(j,k)
\end{equation}
holds for all bounded, symmetric, positive definite functions $g$ \cite[Ch. VIII, Prop. 1.7]{Lig85}, from which \eqref{eq:NegCov} can again be derived. Here, $g$ is positive definite if $\sum_{j,k\in \Z} g(j,k)\beta(j)\beta(k) \geq 0$ when $\sum_{k\in \Z} |\beta(k)|<\infty$ and $\sum_{k\in \Z} \beta(k) = 0$.  In particular, $g(j,k) = 1_{\{j>z, k>z\}}$ is such a function since $\sum_{j,k\in \Z} g(j,k)\beta(j)\beta(k) = \left(\sum_{k\geq z}\beta(k)\right)^2\geq 0$.

 Furthermore, we have the integration-by-parts formula 
\begin{equation}\label{eq:IntByParts}
	U_2(t) - V_2(t) = \int_0^t V_2(t-s)[\sU-\sV]U_s(s)\,ds. 
\end{equation}
For further details, we refer to \cite[Ch. VIII]{Lig85}. 
We use both \eqref{eq:VUcomp} and \eqref{eq:IntByParts} along with duality to obtain the following bound. 

\begin{lemma}\label{CovDual} For any deterministic initial condition $\eta \in \{0,1\}^\Z$, 
\[
	\sC_t(\eta,z) \le 2\sum_{i\ge1}p_i \sum_{k\in \Z} \int_0^t \left( E_k[\eta(\xi_s)] - E_{k+i}[\eta(\xi_s)] \right)^2 P_{k+i}(\xi_{t-s} > z)^2\,ds. 
\]
\end{lemma}

\begin{proof}
Write
\begin{align}
\label{eq:pickupL}
	\sC_t(\eta, z) 	= \underset{j\ne k} {\sum_{j,k > z}}\left[ U_2(t) - V_2(t) \right] \eta(j)\eta(k) 
	= \underset{j\ne k} {\sum_{j,k > z}} \int_0^t V_2(t-s)\left[\sU-\sV\right] U_2(s) \eta(j)\eta(k)\,ds. 
\end{align}
Using \eqref{eq:ExPairgen} and \eqref{eq:RWgen}, we compute for $j \neq k$, using the symmetry $p_{j-k}=p_{k-j}$, that 
\begin{align*}
	\left[\sU-\sV\right]U_2(t)\eta(j)\eta(k) &= p_{k-j}U_2(s)\left[\eta(j)\eta(j)+\eta(k)\eta(k)-2\eta(j)\eta(k)\right] \\
	&=p_{k-j}\left(E_j[\eta(\xi_s)] - E_k[\eta(\xi_s)]\right)^2. 
\end{align*}
Then,
we have
\begin{align*}
	&\underset{k\ne j} {\sum_{j,k > z}} \int_0^t V_2(t-s)\left[\sU-\sV\right] U_2(s) \eta(j)\eta(k)\,ds\\
	&= \int_0^t  \sum_{k \neq j} 1_{\{j>z,k>z\}}V_2(t-s)p_{k-j}\left(E_j[\eta(\xi_s)] - E_k[\eta(\xi_s)]\right)^2\,ds \\
	&= \int_0^t \sum_{k\neq j}p_{k-j}\left(E_j[\eta(\xi_s)] - E_k[\eta(\xi_s)]\right)^2V_2(t-s)1_{\{j>z,k>z\}}\,ds \\
	&\le \int_0^t \sum_{k\neq j}p_{k-j}\left(E_j[\eta(\xi_s)] - E_k[\eta(\xi_s)]\right)^2U_2(t-s)1_{\{j>z,k>z\}}\,ds\\
	&=\int_0^t \sum_{k\neq j}p_{k-j}\left(E_j[\eta(\xi_s)] - E_k[\eta(\xi_s)]\right)^2P_{j}(\xi_{t-s}>z)P_{k}(\xi_{t-s}>z)\,ds\\
	&\le 2\sum_{i\ge1}p_i\sum_{k\in\Z}\int_0^t \left(E_k[\eta(\xi_s)] - E_{k+i}[\eta(\xi_s)]\right)^2P_{k+i}(\xi_{t-s}> z)^2\,ds. 
\end{align*}
Here, we used that $V_2$ is a symmetric operator in the second equality, and $(j,k)\mapsto 1_{\{j>z,k>z\}}$ is positive definite and \eqref{eq:VUcomp} for the next inequality.  For the last inequality, we split over $j>k$ and $j<k$, noting 
$P_{\ell}(\xi_u>z) =P_0(\xi_u> z-\ell)\leq P_0(\xi_u> z-m)=P_{m}(\xi_u>z)$ when $\ell \leq m$ and the symmetry $p_{k-j}=p_{j-k}$.
\end{proof}

\section{Proofs of main theorems}\label{proofs}

As we note in Subsection \ref{prooftech}, we prove 
Theorems \ref{main} and \ref{mainLthm} by showing that $N_t$ as defined in \eqref{eq:N} converges to the required Poisson distribution. 
This requires verification of conditions (i)--(iii) in Lemma \ref{liggettpoisson} for the appropriate initial condition. 
This is done in four steps. 
Let $\nu$ be as in Theorem \ref{main} and $\nu_L$ be as in Theorem \ref{mainLthm}. Also let $a_t$ and $b_t$ denote the sequences in \eqref{eq:ArrScaling}, and let $a_t^{(L)}$ and $b_t^{(L)}$ denote the sequences in \eqref{eq:Lscaling}. Fix $x \in \R$. 

{\bf Step 1.}
For (i) in Lemma \ref{liggettpoisson}, we have that when $z = \sigma b_t(x+a_t)$, 
\begin{equation}\label{eq:fullmean}
	\E_\nu[N_t] \to \sigma\bar\rho e^{-x}, \qquad t\to\infty. 
\end{equation}
Furthermore, 
\begin{equation}\label{eq:fastLmean}
	\E_{\nu_L}[N_t] \to \sigma\bar\rho(1-e^{-c/\sigma})e^{-x}, \qquad t\to\infty, 
\end{equation}
when $Lt^{-1/2}(\log t)^{1/2} \to c \in (0,\infty]$. 
When $z = \sigma b_t^{(L)}(x + a_t^{(L)})$, we have
\begin{equation}\label{eq:slowLmean}
	\E_{\nu_L}[N_t] \to \bar\rho e^{-x}, \qquad t\to\infty, 
\end{equation}
provided $L \to \infty$ and $L = o(t^{1/2}(\log t)^{-1/2})$. 
Proofs of these mean convergence results are done in Section \ref{meancomp} by comparing to the case when $\rho_k = 1$ for all $k \le 0$. In particular, \eqref{eq:fullmean} is shown in Theorem \ref{ArrMean}, and \eqref{eq:fastLmean} and \eqref{eq:slowLmean} are shown in Theorem \ref{Lthm}.

{\bf Step 2.} 
Next, we have the following result, which verifies (ii) in  Lemma \ref{liggettpoisson} in the context of both Theorems \ref{main} and \ref{mainLthm}. It is proved at the end of this section. 
\begin{lemma}\label{SoS} 
Let $\nu$ be any Bernoulli initial condition with $\nu(\eta(k) = 1) = 0$ for $k > 0$. If $z = \sigma b_t(x+a_t)$ or $z = \sigma b_t^{(L)}(x+a_t^{(L)})$, then 
\[
	\sum_{k > z} \left( \E_\nu\left[ \eta_t(k) \right] \right)^2 = O\left( e^{-z^2/(2\sigma^2t)} \E_\nu[N_t] \right). 
\]
In particular, when the limit of $\E_\nu[N_t]$ exists and is proportional to $e^{-x}$ and $z = \sigma b_t(x+a_t)$, 
the sum of squares term is of order 
\[
	e^{-z^2/(2\sigma^2 t)} = O\left( \frac{e^{-2x}\log t}{\sqrt{t}} \right), 
\]
and when $z = \sigma b_t^{(L)}(x+a_t^{(L)})$, 
it is of order 
\[
	e^{-z^2/(2\sigma^2 t)} = O\left(\frac{e^{-2x}\sqrt{\log L}}{L} \right). 
\]
\end{lemma}

{\bf Step 3.}
The final step to proving Theorems \ref{main} and \ref{mainLthm} is to verify (iii) in Lemma \ref{liggettpoisson}. For this we use Lemma \ref{fullstepdom}, which says that it suffices to consider deterministic initial profiles. In the context of Theorem \ref{main} and Theorem \ref{mainLthm}(a), we want to use $\eta(k) = 1$ for $k \le 0$ and $\eta(k) = 0$ for $k > 0$. 
In Section \ref{covcomp} (Theorem \ref{fullcov}), when $z = \sigma b_t(x + a_t)$, we show that 
\begin{equation}\label{eq:cov0}
	\sC_t(\eta, z) \to 0, \qquad t\to\infty.  
\end{equation}

For Theorem \ref{mainLthm}(b), we consider the deterministic initial condition $\eta_L(k) = 1$ for $k \in \{-L, \ldots, 0\}$ and $\eta_L(k) = 0$ otherwise. 
When $z = \sigma b_t^{(L)}(x+a_t^{(L)})$, 
\begin{equation}\label{eq:Lcov0}
	\sC_t(\eta_L,z) \to 0, \qquad t\to\infty, 
\end{equation}
by Theorem \ref{Lcovcomp}. Rates of convergence in \eqref{eq:cov0} and \eqref{eq:Lcov0} are provided in Section \ref{covcomp}.

{\bf Step 4.} Lastly, we put Steps 1--3 together to complete the proofs. 
In particular, by Lemma \ref{liggettpoisson}, Lemmas \ref{fullstepdom} and \ref{SoS} along with \eqref{eq:fullmean} and \eqref{eq:cov0} prove Theorem \ref{main}. Lemmas \ref{liggettpoisson}, \ref{fullstepdom}, and \ref{SoS} with \eqref{eq:fastLmean} and \eqref{eq:cov0} prove Theorem \ref{mainLthm}(a), and with \eqref{eq:slowLmean} and \eqref{eq:Lcov0} prove Theorem \ref{mainLthm}(b). \hfill $\square$

\medskip
\medskip
We finish this section with a proof of Lemma \ref{SoS}. 

\begin{proof}[Proof of Lemma \ref{SoS}] Recall that we may write $\eta_t(k)$ in terms of the stirring variables as 
\[
	\eta_t(k) = \sum_{i \le 0} \eta_0(i) 1_{\{ \xi_i(t) = k\}}, 
\]
which gives
\[
	\E_\nu[\eta_t(k)]  \le 
	\sum_{i\le0} P_{i}(\xi_t = k) = P_0(\xi_t \ge k). 
\] 
Since $\E_\nu[N_t] = \sum_{k>z} \E_\nu[\eta_t(k)]$, we then have 
\begin{align*}
	\sum_{k > z} \left( \E_\nu\left[ \eta_t(k) \right] \right)^2 
	&\le P_0(\xi_t > z)\E_\nu[N_t] = \E_\nu[N_t] \exp\left(-\frac{z^2}{2\sigma^2 t} + O(1)\right), 
\end{align*}
using Lemma \ref{RWchern}. 
\end{proof}

\section{Computation of the mean}\label{meancomp}

Here we show that the mean of $N_t$ converges to the parameter of the appropriate Poisson distribution in the contexts of Theorems \ref{main} and \ref{mainLthm}. We recall the following conventions: $x_t \sim y_t$ means $\lim_{t\to\infty} x_t/y_t = $, $x_t = O(y_t)$ means $x_t \le Cy_t$ for some $C > 0$ and large $t$, and $x_t=o(y_t)$ when $\lim_{t\to\infty}x_t/y_t =0$.  Recall also the representation \eqref{eq:stirringN}.

\begin{theorem}\label{ArrMean} Let $\eta \in \{0,1\}^\Z$ be defined by $\eta(k) = 1$ for $k \le 0$ and $\eta(k)=0$ otherwise. In the setup of Theorem \ref{main}, 
\begin{equation}\label{eq:nutoeta}
	\E_\nu[N_t] = \bar\rho\E_\eta[N_t] + o(1), \qquad t \to \infty, 
\end{equation}
and 
\[
	\lim_{t\to\infty} \E_\eta[N_t] = \sigma e^{-x}. 
\]
\end{theorem}

\begin{proof} Since $\rho_j = \rho_{j-m}$ for each $j<0$ and $\rho_0=1$, we have 
\begin{align*}
	\E_\nu[N_t] &= \sum_{i \le 0} \rho_i P_i(\xi_t > z) \\
	&= P_0(\xi_t > z) + \sum_{j=1}^m \rho_{-j} \sum_{i \ge 0} P_{-im-j}(\xi_t > z) \\
	&= P_0(\xi_t > z)  + \sum_{j=1}^m \rho_{-j} \sum_{i \ge 0} P_0(\xi_t -z - j > im) \\
	&= \frac{1}{m}\sum_{j=1}^m \rho_{-j} E_0\left[(\xi_t - z - j)^+\right] + o(1). 
\end{align*}
A simpler calculation gives 
\begin{equation}\label{eq:Eeta}
	\E_\eta[N_t] = E_0\left[(\xi_t-z)^+\right] + o(1). 
\end{equation}
Thus, 
\begin{align*}
	\left| \bar\rho \E_\eta[N_t] - \E_\nu[N_t] \right| &= \frac{1}{m}\sum_{j=1}^m \rho_{-j} \left( E_0\left[(\xi_t-z)^+\right] - E_0\left[(\xi_t - z - j)^+\right] \right) + o(1) \\
	&= \frac{1}{m} \sum_{j=1}^m \rho_{-j} \left( E_0\left[ (\xi_t-z)1_{\{0 < \xi_t-z \le j\}} \right] + jP_0(\xi_t > z+j) \right) + o(1) \\
	&\le mP_0(\xi_t > z) + o(1) = o(1),
\end{align*}
since $z \sim \sigma\sqrt{t\log t}$. This shows \eqref{eq:nutoeta}. 

Now we show that $\E_\eta[N_t] \to \sigma e^{-x}$. 
Starting from \eqref{eq:Eeta}, let $w = z/(\sigma\sqrt{t})$, and note that $w = O(\sqrt{\log t})$.  
Recalling $X \sim \sN(0,1)$, 
\begin{align}\nonumber
	&E_0\left[ (\xi_t - z)^+\right] = \sigma \sqrt{t} \int_w^\infty P_0(\xi_t >\sigma u\sqrt{t})\,du \\ \label{eq:O1}
	&= \sigma\sqrt{t} \int_w^{\log t} P_0(\xi_t > \sigma u\sqrt{t})\,du + o(1)\\ \nonumber 
	&= \sigma\sqrt{t}\int_w^{\log t} P(X > u)\,du + \sigma\sqrt{t} \int_w^{\log t}P(X > u) \left[ \frac{P_0(\xi_t >\sigma u\sqrt{t})}{P(X>u)}-1\right] + o(1)\\ \label{eq:O2}
	&= \sigma\sqrt{t} \int_w^\infty P(X > u)\,du + O\left( \int_w^\infty u^3P(X > u)\,du\right) +o(1) \\ 
	&= \sigma \sqrt{t} E\left[(X-w)^+\right] + o(1). 
\end{align}
Above, \eqref{eq:O1} is justified by Lemma \ref{RWchern} and 
\[
	\sqrt{t} \int_{\log t}^\infty P_0(\xi_t > \sigma u \sqrt{t})\,du = \sqrt{t} \int_{\log t}^\infty e^{-u^2/2 + O(1)}\,du = O\left(\sqrt{t}e^{-(\log t)^2/2}\right).
\]
\eqref{eq:O2} follows from Lemma \ref{FellerNormal}
as well as from Lemma \ref{flemma}, which gives 
\[
	\sqrt{t}\int_{\log t}^\infty P(X>u)\,du = \sqrt{t}E\left[(X- \log t)^+\right] = O\left(\frac{\sqrt{t}e^{-(\log t)^2/2}}{(\log t)^2}\right) . 
\]
Note that for $z = \sigma b_t(x+a_t)$, 
\begin{align*}
	e^{-w^2/2} &= \frac{\sqrt{2\pi} e^{-x+o(1)}\log t}{\sqrt{t}}, 
\end{align*}
so by Lemma \ref{flemma}, 
\begin{align*}
	\sigma \sqrt{t}E\left[(X-w)^+\right] &= \frac{\sigma\sqrt{t}\varphi(w)}{w^2} + O\left(\frac{\sqrt{t}\varphi(w)}{w^4}\right) = \sigma e^{-x} \left(1 + o(1)\right). \qedhere 
\end{align*}
\end{proof}

The next theorem concerns convergence of the mean of $N_t$ for the $L$-step setting.

\begin{theorem}\label{Lthm} For each $L > 0$, define $\eta_L \in \{0,1\}^\Z$ by $\eta_L(k) = 1$ for $-L \le k \le 0$ and $\eta_L(k) = 0$ otherwise. With the assumptions of Theorem \ref{mainLthm}, 
\begin{equation}\label{eq:nutoetaL}
	\E_{\nu_L}[N_t] = \bar\rho \E_{\eta_L}[N_t] + o(1), \qquad t\to\infty. 
\end{equation}
Furthermore, 
\begin{enumerate}

	\item[(a)] If $L \sqrt{\frac{\log t}{t}} \to c \in (0, \infty]$ and $z = \sigma b_t(x+a_t)$ for $a_t$ and $b_t$ as in \eqref{eq:ArrScaling}, then 
	\[
		\lim_{t\to\infty}\E_{\eta_L}[N_t] = \sigma(1 - e^{-c/\sigma})e^{-x}. 
	\]

	\item[(b)] If $L\to\infty$ such that $L \sqrt{\frac{\log t}{t}} \to 0$ and $z = \sigma b_t(x+a_t)$ for $a_t=a^{(L)}$ and $b_t=b_t^{(L)}$ as in \eqref{eq:Lscaling}, then 
	\[
		\lim_{t\to\infty}\E_{\eta_L}[N_t] = e^{-x}. 
	\]

\end{enumerate}
\end{theorem}

\begin{proof} Proving \eqref{eq:nutoetaL} follows the same argument as for \eqref{eq:nutoeta} and is omitted. Now consider part (a). For $f(u) = \varphi(u) - uP(X>u)$, Lemma \ref{flemma} gives 
\begin{align*}
	&\sqrt{t}\left(f\left(\frac{z}{\sigma\sqrt{t}}\right) - f\left(\frac{L+z}{\sigma\sqrt{t}}\right)\right) \\
	&= \frac{\sqrt{t}\varphi\left(z/(\sigma\sqrt{t})\right)}{\left(z/(\sigma\sqrt{t})\right)^2} \left( 1 - \frac{e^{-\left(L^2/2 + Lz\right)/(\sigma^2 t)}}{\left(1 + L/z\right)^2} \right) + O\left( \frac{\sqrt{t}\varphi\left(z/(\sigma\sqrt{t})\right)}{(z/\sqrt{t})^4}\right) \\
	&= \frac{\sigma^2 t^{3/2} e^{-z^2/(2\sigma^2 t)}}{z^2\sqrt{2\pi}}\left( 1 - \frac{e^{-\left(L^2/2 + Lz\right)/(\sigma^2 t)}}{\left(1 + L/z\right)^2} \right) + O\left( \frac{t^{5/2}e^{-z^2/(2\sigma^2t)}}{z^4}\right). 
\end{align*}
We have that 
\[
	\frac{z^2}{2\sigma^2 t} = \frac{1}{2}\log t + x - \log\sqrt{2\pi} - \log\log t + o(1), \qquad t\to\infty, 
\]
and hence
\[
	\frac{\sigma^2 t^{3/2} e^{-z^2/(2\sigma^2 t)}}{z^2\sqrt{2\pi}} \to e^{-x}, 
\]
as well as 
\[
	\frac{t^{5/2}e^{-z^2/(2\sigma^2t)}}{z^4} = O\left( \frac{te^{-x}}{z^2} \right) = O\left( \frac{e^{-x}}{\log t}\right).  
\]
Then by Lemma \ref{Nasym} below, 
\begin{align*}
	\frac{\E_{\eta_L}[N_t]}{\sigma} &\sim \sqrt{t}\left(f\left(\frac{z}{\sigma\sqrt{t}}\right) - f\left(\frac{L+z}{\sigma\sqrt{t}}\right) \right) 
	\sim e^{-x} \left( 1 -\frac{e^{-\left(L^2/2 + Lz\right)/(\sigma^2 t)}}{\left(1 + L/z\right)^2} \right). 
\end{align*}
Now note, by the assumption $Lt^{-1/2}(\log t)^{1/2} \to c > 0$, 
\[
	\frac{Lz}{\sigma^2 t} = \frac{L}{\sigma} \sqrt{\frac{\log t}{t}} \left(1 + \frac{x-\log\sqrt{2\pi}-\log\log t}{\log t}\right) \to \frac{c}{\sigma}. 
\]
If $c = \infty$, then this shows $E[N_t]/\sigma \to e^{-x}$. If $c < \infty$, then $L/z = o(1)$ and $L^2/t = o(1)$, and hence 
\[
	\frac{\E_{\eta_L}[N_t]}{\sigma} \sim e^{-x} \left( 1 - e^{-Lz/(\sigma^2 t)} \right) \to e^{-x}\left(1 - e^{-c/\sigma}\right), 
\]
which proves part (a). 

For part (b), as $z \sim \sigma\sqrt{t\log L^2}$, we have $L/z = o(1)$, $L^2/t = o(1)$, and $Lz/t = o(1)$, and Lemma \ref{flemma} gives
\begin{align*}
	&\sqrt{t}\left(f\left(\frac{z}{\sigma\sqrt{t}}\right) - f\left(\frac{L+z}{\sigma\sqrt{t}}\right)\right) \\
	&= \frac{\sqrt{t}\varphi\left(z/(\sigma\sqrt{t})\right)}{\left(z/(\sigma\sqrt{t})\right)^2} \left( 1 - \frac{e^{-\left(L^2/2 + Lz\right)/(\sigma^2 t)}}{\left(1 + L/z\right)^2} \right) + O\left( \frac{L\varphi\left(z/(\sigma\sqrt{t})\right)}{(z/\sqrt{t})^3}\right) \\
	&= \frac{\sigma^2 t^{3/2} e^{-z^2/(2\sigma^2 t)}}{z^2\sqrt{2\pi}}\left( 1 - \frac{e^{-\left(L^2/2 + Lz\right)/(\sigma^2 t)}}{\left(1 + L/z\right)^2} \right) + O\left( \frac{Lt^{3/2}e^{-z^2/(2\sigma^2t)}}{z^3}\right). 
\end{align*}
Since 
\[
	\frac{z^2}{2\sigma^2t} =  \log L + x - \log\sqrt{2\pi} - \frac{1}{2}\log\log L^2 + o(1), 
\]
we have that 
\begin{equation}\label{eq:sqrtphi}
	\frac{\sigma^2 t^{3/2} e^{-z^2/(2\sigma^2 t)}}{z^2\sqrt{2\pi}} \sim \frac{\sqrt{t}e^{-x}}{L\sqrt{\log L^2}}. 
\end{equation}
Furthermore, since $z \sim \sigma\sqrt{t\log L^2}$, 
\[
	\frac{Lt^{3/2}e^{-z^2/(2\sigma^2t)}}{z^3} = O\left( \frac{\sqrt{t}e^{-x}}{L\sqrt{\log L^2}}\cdot \frac{L}{z} \right) = O\left(\frac{e^{-x}}{\log L}\right). 
\]
From this, \eqref{eq:sqrtphi}, $z \sim \sigma\sqrt{t\log L^2}$, $L \to \infty$, and $L = o(\sqrt{t})$, it follows that 
\begin{align*}
	\frac{\E_{\eta_L}[N_t]}{\sigma} &\sim \sqrt{t}\left(f\left(\frac{z}{\sigma\sqrt{t}}\right) - f\left(\frac{L+z}{\sigma\sqrt{t}}\right) \right) \\
	&\sim \frac{\sqrt{t}e^{-x}}{L\sqrt{\log L^2}} \left( 1 - e^{-\left(L^2/2 + Lz\right)/(\sigma^2 t)} \right) \\
	&= \frac{\sqrt{t}e^{-x}}{L\sqrt{\log L^2}}\left(\frac{L^2}{2\sigma^2t} + \frac{L\sqrt{\log L^2}}{\sigma \sqrt{t}}\right) + O\left( \frac{e^{-x}L^3}{t^{3/2}\sqrt{\log L}} + \frac{e^{-x}Lz^2}{t^{3/2}} \right)\\
	&= \frac{e^{-x}}{\sigma} + o(1), 
\end{align*}
which completes part (b). 
\end{proof}

We now prove a lemma used to show the previous result. 

\begin{lemma}\label{Nasym} 
Let $\eta_L$ be as in Theorem \ref{Lthm}, and recall $z=\sigma b_t(x+a_t)$. For $X \sim \sN(0,1)$, let $f(u) = E[(X-u)^+] = \varphi(u) - uP(X > u)$, where $\varphi(u) = (2\pi)^{-1/2}e^{-u^2/2}$. If either 
\begin{enumerate}[(i)]
\item $a_t$ and $b_t$ are as in \eqref{eq:ArrScaling}, or

\item $a_t=a^{(L)}$ and $b_t=b_t^{(L)}$ are as in \eqref{eq:Lscaling} and $L\sqrt{\frac{\log t}{t}} \to 0$ as $t \to \infty$, 
\end{enumerate}
then, 
\begin{equation}\label{eq:fdiff}
	\E_{\eta_L}[N_t]\sim \sigma\sqrt{t}\left(f\left(\frac{z}{\sigma\sqrt{t}}\right) - f\left(\frac{L+z}{\sigma\sqrt{t}}\right)\right), \qquad t \to \infty. 
\end{equation}
\end{lemma}

\begin{proof} 
For notational convenience, let $w = z/(\sigma \sqrt{t})$ and $M = L/(\sigma\sqrt{t})$.  
Using \eqref{eq:stirringN}, we have
\begin{align}\nonumber
	\E_{\eta_L}[N_t] &= \sum_{-L \le i \le 0} P_i(\xi_t > z) 
	= \sum_{i=0}^L P_0(\xi_t - z > i) \\\nonumber
	&= \int_0^L P_0(\xi_t - z > u)\,du + o(1) \\\nonumber
	&= \sigma\sqrt{t} \int_{w}^{w+M} P_0\left( \frac{\xi_t}{\sigma \sqrt{t}} > u\right)\,du + o(1) \\\label{eq:whereileftoff}
	&= \sigma \sqrt{t} \int_{w}^{w+M} P\left( X > u\right)\,du + o(1),
\end{align}
where \eqref{eq:whereileftoff} is justified by Lemma \ref{FellerNormal} as follows. First note that $w \to \infty$ and $w = o(t^{1/6})$ in either case (i) or (ii) when $z \sim \sigma \sqrt{t\log t}$ or $z\sim \sigma \sqrt{t\log L^2}$, respectively. If $M = o(t^{1/6})$, then for some $C > 0$ and large enough $t$, 
\begin{equation}\label{eq:normjust}
\begin{aligned}
	&\left| \int_{w}^{w+M} P_0(\xi_t > \sigma u\sqrt{t})\,du - \int_{w}^{w+M} P(X > u)\,du \right| \\ 
		&\le \int_{w}^{w+M} P(X > u) \left| 1 - \frac{P_0(\xi_t > \sigma u\sqrt{t})}{P(X > u)}\right|\,du \le \frac{C}{\sqrt{t}} \int_{w}^\infty u^3 P(X > u)\,du	
	\to 0.
\end{aligned}
\end{equation}
Otherwise, by the proof of Theorem \ref{ArrMean} starting at \eqref{eq:O1}, we have 
\begin{align*}
	\sqrt{t} \int_{w}^{w+M} P_0\left( \frac{\xi_t}{\sigma \sqrt{t}} > u\right)\,du &\sim \sqrt{t} \int_{w}^{\log t} P_0\left( \frac{\xi_t}{\sigma \sqrt{t}} > u\right)\,du \\
	& \sim \sqrt{t} \int_{w}^{\log t} P\left( X > u\right)\,du \sim \sqrt{t} \int_{w}^{w+M} P\left( X > u\right)\,du. 
\end{align*}
This proves \eqref{eq:whereileftoff}. 

To finish showing \eqref{eq:fdiff}, we have from \eqref{eq:whereileftoff} that
\begin{align*}
	\frac{\E_{\eta_L}[N_t]}{\sigma\sqrt{t}} &\sim \int_{w}^\infty P(X > u) \,du- \int_{w+M}^\infty P(X > u)\,du \\
	&=  E\left[ (X - w)1_{\{X > w\}}\right] - E\left[ (X - (w+M))1_{\{X > w+M\}}\right] 
	=  f(w) - f(w+M). \qedhere
\end{align*}
\end{proof}

\section{Bounds for the covariance}\label{covcomp}

Here we show \eqref{eq:cov0} and \eqref{eq:Lcov0} separately. While the proof schemes for the full step and $L$-step are similar, in the $L$-step setting, we need to use the shape of the initial profile to accomodate the $L$-dependent scaling in \eqref{eq:Lscaling}. 

\subsection{Proof of (\ref{eq:cov0})}
First, we consider the full step deterministic initial profile $\eta \in \{0,1\}^\Z$ defined by $\eta(k) = 1$ if and only if $k \le 0$. Suppose the assumptions of Theorem \ref{main}. In particular, $z = \sigma b_t(x+a_t)$ for the scaling in \eqref{eq:ArrScaling}.
Lemma \ref{CovDual} gives 
\[
	\sC_t(\eta,z) \le 2\sum_{i\ge1}p_i \sum_{k\in \Z} \int_0^t \left( E_k[\eta(\xi_s)] - E_{k+i}[\eta(\xi_s)] \right)^2 P_{k+i}(\xi_{t-s} > z)^2\,ds. 
\]
For $k\in \Z$ and $i\geq 1$, and by symmetry of the random walk, we compute 
\begin{align*}
	\left(E_k[\eta(\xi_s)] - E_{k+i}[\eta(\xi_s)]\right)^2 &= \left(P_k(\xi_s \le 0) - P_{k+i}(\xi_s\le 0)\right)^2 = P_0(-k-i < \xi_s \leq -k)^2 \\
	&= P_0(k \le \xi_s < k+i)^2 \le i\sum_{j=0}^{i-1}P_0(\xi_s=k+j)^2, 
\end{align*}
and so 
\begin{align}\nonumber
	\sC_t(\eta,z) 
	&\le 2\sum_{i\ge1}ip_i\sum_{k\in \Z}\sum_{j=0}^{i-1}\int_0^t P_0(\xi_s=k+j)^2P_{k+i}(\xi_{t-s}> z)^2\,ds \\ \nonumber
	&= 2\sum_{i\ge1}ip_i\sum_{k\in \Z}\sum_{j=0}^{i-1} \int_0^t P_0(\xi_s=k)^2P_{k+i-j}(\xi_{t-s}> z)^2\,ds \\ \label{eq:diffzs}
	&\le 2\sum_{i\ge1}i^2p_i \sum_{k\in \Z} \int_0^t P_0(\xi_s=k)^2P_0(\xi_{t-s} > z-k-i)^2\,ds. 
\end{align}
We then have the following. 

\begin{theorem}\label{fullcov} If $\{p_i\}$ is finite range, then 
\[
	\sC_t(\eta, z) = O\left(    \frac{e^{-2x} (\log t)^2}{\sqrt{t}}       \right), \qquad t\to\infty. 
\]
Otherwise, there is $\alpha > 0$ such that for any $\eps \in(0,1/2)$, 
\[
	\sC_t(\eta, z) = O\left(    \frac{\left(e^{-x} \log t\right)^{2(1-\eps)}  }{t^{1/2-\eps}}  + e^{-\alpha \eps \sqrt{t\log t}}    \right), \qquad t\to\infty. 
\]
\end{theorem}

\begin{proof} In Lemma \ref{covlemmafull} below, we show that when $\tilde z = O(\sqrt{t \log t})$, then 
\[	
	\sum_{k\in\Z} \int_0^t P_0(\xi_s=k)^2P_0(\xi_{t-s} \ge \tilde z - k)^2\,ds = O\left(\sqrt{t} \exp\left(-\frac{\tilde z^2}{\sigma^2 t} \right)\right), \qquad t\to\infty. 
\]
In the finite range case, there is some $r > 0$ such that $p_i = 0$ for $i > r$. Then since $-(z-r)^2/t = -z^2/t + o(1)$, we let $\tilde z = z - r$ to obtain from \eqref{eq:diffzs} that 
\begin{align}\nonumber
	\sC_t(\eta, z) &\le \sigma^2  \sum_{k\in \Z} \int_0^t P_0(\xi_s=k)^2P_0(\xi_{t-s} > z-k-r)^2\,ds \\ \label{eq:frcov}
	&= O\left(\sqrt{t}e^{-z^2/(\sigma^2 t)} \right) = O\left(    \frac{e^{-2x} (\log t)^2}{\sqrt{t}}       \right). 
\end{align}
Otherwise, by Condition \ref{mgf}, there is some $\alpha > 0$ such that  for any $\delta \in (0,1)$,
\begin{align}\nonumber
	\sC_t(\eta, z) &\le 2\sum_{i \le \delta z} i^2p_i \sum_{k\in \Z} \int_0^t P_0(\xi_s=k)^2P_0(\xi_{t-s} > z-k-i)^2\,ds + 2t\sum_{i > \delta z} i^2p_i \\\nonumber
	&= \sigma^2 \sum_{k\in \Z} \int_0^t P_0(\xi_s=k)^2P_0(\xi_{t-s} > (1-\delta)z-k)^2\,ds + O\left(e^{-\alpha\delta z}\right) \\\label{eq:gencov}
	&= O\left( \sqrt{t} e^{-(1-\delta)^2z/(\sigma^2 t)} + e^{-\alpha \delta z} \right) = O\left( \frac{\left(e^{-x}\log t\right)^{2(1-\delta)^2}}{t^{(1-\delta)^2 - 1/2}}  + e^{-\alpha \delta z} \right). 
\end{align}
Now choose $\delta$ so that $1-\eps = (1-\delta)^2$, and possibly alter $\alpha$ since then $\delta \ge \eps/2$. 
\end{proof}

We now prove the lemma used to show the previous result. 

\begin{lemma}\label{covlemmafull}
If $z = O(\sqrt{t \log t})$, then 
\[	
	\sum_{k\in\Z} \int_0^t P_0(\xi_s=k)^2P_0(\xi_{t-s} \ge z - k)^2\,ds = O\left(\sqrt{t} e^{-z^2/(\sigma^2 t)}\right), \qquad t\to\infty. 
\]
\end{lemma}

\begin{proof} Write 
\begin{align}\nonumber
	&\sum_{k \in \Z} \int_0^t P_0(\xi_s = k)^2P_0(\xi_{t-s} \ge z - k )^2\,ds \\\label{eq:pt4}
	&\le  \int_{\sqrt{t}}^{t-z^{4/3}} \sum_{|k| \le s^{2/3}\wedge z}P_0(\xi_s = k)^2P_0(\xi_{t-s} \ge z - k )^2\,ds \\\label{eq:pt1}
	&\quad + \sum_{k \in \Z} \int_0^{\sqrt{t}} P_0(\xi_s = k)^2P_0(\xi_{t-s} \ge z - k )^2\,ds 
	\\\label{eq:pt2}
	&\quad + \int_{\sqrt{t}}^t \sum_{|k| > s^{2/3}} P_0(\xi_s = k)^2P_0(\xi_{t-s} \ge z - k )^2\,ds \\ \label{eq:pt3}
	&\quad + \int_{0}^t \sum_{ z < |k| \le s^{2/3}} P_0(\xi_s = k)^2P_0(\xi_{t-s} \ge z - k )^2\,ds 		\\ \label{eq:pt5}
	&\quad + \int_{t-z^{4/3}}^t \sum_{ |k| \le s^{2/3}\wedge z}P_0(\xi_s = k)^2P_0(\xi_{t-s} \ge z - k )^2\,ds. 
\end{align}
We proceed by bounding \eqref{eq:pt4}--\eqref{eq:pt5} in separate steps. We note that the $s^{2/3}$ is chosen for simplicity (and so that we may apply Lemma \ref{localCLT}), although there is room for a smaller exponent. 

{\bf Step 1.} We prove that \eqref{eq:pt4} is of order $\sqrt{t}e^{-z^2/(\sigma^2 t)}$. 
When $|k| \le s^{2/3} \wedge z$ and $s \le t - z^{4/3}$, so that $t-s \ge z^{4/3}$, we have from Lemma \ref{RWchern} that
\begin{align*}
	P(\xi_s =k)^2P(\xi_{t-s} \ge z-k)^2 &= (2\pi s)^{-1}\exp\left(-\frac{k^2}{\sigma^2s} - \frac{(z-k)^2}{\sigma^2(t-s)} + O\left(\frac{|k|^3}{s^2} + \frac{z^4}{(t-s)^3}\right)\right) \\
	&= (2\pi s)^{-1} \exp\left( -\frac{z^2}{\sigma^2t} - \frac{t}{\sigma^2s(t-s)}\left(k - \frac{sz}{t}\right)^2 + O(1)\right) \\
	&= (2\pi s)^{-1} \exp\left( -\frac{z^2}{\sigma^2t} - \frac{(k-c)^2}{\sigma^2s} + O(1)\right), 
\end{align*}
for $c = sz/t$. Then, when $s \le t-z^{4/3}$, 
\begin{align*}
	&\sum_{|k|\le s^{2/3}\wedge z}P_0(\xi_s = k)^2P_0(\xi_{t-s} \ge z - k )^2\\
	&= \frac{e^{-z^2/(\sigma^2t)} }{2\sqrt{\pi s}} \left( 1 + \int_{-\infty}^\infty (\pi s)^{-1/2} \exp\left(- \frac{(u-c)^2}{\sigma^2s} + O(1)\right)\,du\right) = O\left(s^{-1/2} e^{-z^2/(\sigma^2 t)} \right). 
\end{align*}
It follows that 
\[
	\int_{\sqrt{t}}^{t-z^{4/3}} \sum_{|k|\le s^{2/3}\wedge z}P_0(\xi_s = k)^2P_0(\xi_{t-s} \ge z - k )^2\,ds = O\left(\sqrt{t} e^{-z^2/(\sigma^2t)}\right). 
\]

{\bf Step 2.} We show that \eqref{eq:pt1} is of order $\sqrt{t}e^{-z^2/(\sigma^2 t)}$ as well. Indeed, 
the Markov property and Lemma \ref{RWchern} imply 
\begin{align*}
	&\sum_{k \in \Z} \int_0^{\sqrt{t}} P_0(\xi_s = k)^2P_0(\xi_{t-s} \ge z - k )^2\,ds \\
	&\le \int_0^{\sqrt{t}} \left[\sum_{k \in \Z} P_0(\xi_s = k)P_k(\xi_{t-s} \ge z)\right]^2\,ds \\
	&= \sqrt{t}P_0(\xi_t \ge z)^2 = \sqrt{t}\exp\left(-\frac{z^2}{\sigma^2 t} + O\left(\frac{z^{4}}{t^3}\right)\right) = O\left(\sqrt{t}e^{-z^2/(\sigma^2 t)}\right), 
\end{align*}
since $z = O(\sqrt{t\log t})$. 

{\bf Step 3.} \eqref{eq:pt2} is of order $t^{2/3} e^{-t^{1/6}/\sigma^2}$. 
Again by Lemma \ref{RWchern}, 
\begin{align*}
	&\int_{\sqrt{t}}^t \sum_{|k| > s^{2/3}} P_0(\xi_s = k)^2P_0(\xi_{t-s} \ge z - k )^2\,ds \\
	&\le \int_{\sqrt{t}}^t \sum_{|k| > s^{2/3}} P_0(\xi_s = k)^2\,ds \\
	&\le 2\int_{\sqrt{t}}^\infty P_0(\xi_s > s^{2/3})^2\,ds
	= 2\int_{\sqrt{t}}^\infty \exp\left(-\frac{s^{1/3}}{\sigma^2} + O(1)\right)\,ds 
	= O\left(t^{2/3} e^{-t^{1/6}/\sigma^2}\right). 
\end{align*}

{\bf Step 4.} We show that \eqref{eq:pt3} is of order $\sqrt{t}e^{-z^2/(\sigma^2 t)}$.
For this we use the local central limit theorem (Lemma \ref{localCLT}) to obtain
\begin{align*}
	&\int_{z^{3/2}}^t \sum_{ z < |k| \le s^{2/3}} P_0(\xi_s = k)^2P_0(\xi_{t-s} \ge z - k )^2\,ds \\ 
	&\le 2\int_{0}^t \sum_{ z < k \le s^{2/3}} P_0(\xi_s = k)^2\,ds \\
	&= 2\int_{0}^t \sum_{ z < k \le s^{2/3}} (2\pi s)^{-1} \exp\left(-\frac{k^2}{\sigma^2s} + O\left(\frac{k^3}{s^2}\right)\right)\,ds \\
	&= O\left( \int_0^t s^{-{1/2}}\int_z^\infty \frac{e^{-u^2/(\sigma^2s)}}{\sqrt{s}}\,du\,ds \right) 
	= O\left(\sqrt{t} e^{-z^2/(\sigma^2t)}\right). 
\end{align*}

{\bf Step 5.}
The last term \eqref{eq:pt5} is also of order $\sqrt{t}e^{-z^2/(\sigma^2 t)}$. To show this, let $\eps \in(0,1)$ and write 
\begin{align*}
	&\int_{t-z^{4/3}}^t \sum_{ |k| \le s^{2/3}\wedge z}P_0(\xi_s = k)^2P_0(\xi_{t-s} \ge z - k )^2\,ds\\
	&\le \int_{t-z^{4/3}}^t \sum_{ |k| \le s^{2/3}\wedge (\eps z)}P_0(\xi_s = k)^2P_0(\xi_{t-s} \ge z - k )^2\,ds \\
	&\quad + \int_{t-z^{4/3}}^t \sum_{\eps z < |k| \le s^{2/3} \wedge z} P_0(\xi_s = k)^2\,ds. 
\end{align*}
When $s > t-z^{4/3}$ and $|k| \le \eps z$, 
\[
	P_0(\xi_{t-s} \ge z - k ) \le P_0(\xi_{t-s} \ge (1-\eps)z) \le \frac{\sigma^2(t-s)}{(1-\eps)^2 z^2} \le \frac{\sigma^2}{(1-\eps)^2z^{2/3}}. 
\]
Then we apply Lemma \ref{localCLT} to obtain 
\begin{align*}
	&\int_{t-z^{4/3}}^t \sum_{ |k| \le s^{2/3}\wedge (\eps z)}P_0(\xi_s = k)^2P_0(\xi_{t-s} \ge z - k )^2\,ds \\
	&\le \frac{\sigma^2}{(1-\eps)^4z^{4/3}} \int_{t-z^{4/3}}^t \sum_{ |k| \le s^{2/3}\wedge (\eps z)}P_0(\xi_s = k)^2\,ds \\
	&= \frac{\sigma^2}{(1-\eps)^4z^{4/3}} \int_{t-z^{4/3}}^t \sum_{ |k| \le s^{2/3}\wedge (\eps z)} s^{-1} \exp\left(- \frac{k^2}{\sigma^2s} + O(1)\right)\,ds \\
	&= \frac{\sigma^2}{(1-\eps)^4z^{4/3}} \log\left(\frac{t}{t-z^{4/3}}\right)\left(1 + \int_{-\infty}^\infty \exp\left(-\frac{u^2}{\sigma^2t} + O(1)\right)\,du \right) = O\left(\frac{1}{\sqrt{t}}\right), 
\end{align*}
since 
\[
	 \log\left(\frac{t}{t-z^{4/3}}\right) = \log \left(1 + \frac{z^{4/3}}{t-z^{4/3}}\right) = O\left(\frac{z^{4/3}}{t}\right). 
\]
Furthermore, 
\begin{align*}
	\int_{t-z^{4/3}}^t \sum_{\eps z < |k| \le s^{2/3} \wedge z} P_0(\xi_s = k)^2\,ds &= \int_{t-z^{4/3}}^t \sum_{\eps z < |k| \le s^{2/3} \wedge z} s^{-1}\exp\left(-\frac{k^2}{\sigma^2s} + O(1)\right) \,ds \\
	&= \int_{t-z^{4/3}}^t \frac{1}{\sqrt{s}} \int_{\eps z}^\infty s^{-1/2} \exp\left(-\frac{u^2}{\sigma^2s} + O(1)\right)\,du \,ds\\
	&= O\left( z^{1/3}e^{-\eps^2z^2/(\sigma^2t)}\right). 
\end{align*}
Finally, $z=O(\sqrt{t\log t})$ implies 
\[
	t^{-1/2} + z^{1/3}e^{-\eps^2z^2/(\sigma^2t)} = O\left(\sqrt{t}e^{-z^2/(\sigma^2t)}\right), \qquad t \to \infty. 
\]

Steps 1--5 together complete the proof. 
\end{proof}

\subsection{Proof of  (\ref{eq:Lcov0})}

Now we turn to the deterministic $L$-step initial profile where $\eta_L(k) = 1$ if and only if $-L \le k \le 0$. Suppose the conditions of Theorem \ref{mainLthm} so that $z = \sigma b_t(x + a_t)$ for the scalings $a_t=a_t^{(L)}$, $b_t=b_t^{(L)}$ in \eqref{eq:Lscaling}. 
Note that, for $j < k$, using symmetry of the random walk in the second equality,
\begin{align*}
	\left(E_j[\eta_L(\xi_s)] - E_k[\eta_L(\xi_s)]\right)^2 
	&= \left(P_j(-L\le\xi_s\le0) - P_k(-L\le\xi_s\le0)\right)^2 \\
	&= \left(P_0(k \le \xi_s \le k+L) - P_0(j \le \xi_s \le j+L)\right)^2\\
	&= \left( P_0(\xi_s \ge k) - P_0(\xi_s > k+L) - P_0(\xi_s \ge j) + P_0(\xi_s > j+L)\right)^2 \\
	&=\left(P_0(j \le \xi_s < k) - P_0(j+L+1 \le \xi_s < k+L+1)\right)^2.
\end{align*}
Then, by Lemma \ref{CovDual}, inserting the above estimate in its statement,
 we have 
\begin{align} 
 \label{eq:Lcovstep}
	\sC_t(\eta_L,z)  
	\le 2\sum_{i\ge1} i^2p_i \sum_{k\in\Z} \int_0^t \left[P_0(\xi_s=k)-P_0(\xi_s=k+L+1)\right]^2P_0(\xi_{t-s}> z-k-i)^2\,ds. 
\end{align}
Note that \eqref{eq:Lcovstep} matches \eqref{eq:diffzs} when $L = \infty$, however when $L<\infty$, \eqref{eq:Lcovstep} is stricly smaller. In particular, to accomodate the more slowly-diverging scaling in \eqref{eq:Lscaling} when $L<\infty$, we will need to take advantage of estimates of the difference $P_0(\xi_s=k)-P_0(\xi_s=k+L+1)$, rather than just bound $P_0(\xi_t=k)$ as in the proof of Lemma \ref{covlemmafull}. 

We have the following. 

\begin{theorem}\label{Lcovcomp} Suppose that $L \to \infty$ such that $L = o(\sqrt{t/\log t})$ as $t\to\infty$.
If $\{p_i\}$ is finite range, then 
\[
	\sC_t(\eta_L, z) = O\left(  \frac{(e^{-2x}\vee 1)(\log L)^2}{L}      \right), \qquad t\to\infty. 
\]
Otherwise, there is $\alpha > 0$ so that for any $\eps \in(0,1/2)$, 
\[
	\sC_t(\eta_L, z) = O\left( \frac{(e^{-2(1-\eps)x} \vee 1)(\log L)^2}{L^{1-2\eps}} + e^{-\alpha\eps\sqrt{t\log L}} \right), \qquad t\to\infty. 
\]
\end{theorem}

\begin{proof} From Lemma \ref{LCovrate} below, we have that 
\begin{align*}
	&\sum_{k\in\Z} \int_0^t \left[P_0(\xi_s=k)-P_0(\xi_s=k+L+1)\right]^2P_0(\xi_{t-s} \ge \tilde z-k)^2\,ds \\
	&= O\left( L\log L \cdot e^{-\tilde z^2/(\sigma^2t)} + e^{-\tilde z^2/(2\sigma^2t)} + \frac{(\log L)^2}{L} \right), \qquad t \to\infty, 
\end{align*}
as long as $\tilde z \sim c\sqrt{t\log L}$ for some $c > 0$.

 Now, as in the proof of Theorem \ref{fullcov} (see \eqref{eq:frcov} and \eqref{eq:gencov}), when $\{p_i\}$ is finite range, we have 
\[
	\sC_t(\eta_L, z) = O\left( L\log L \cdot e^{-z^2/(\sigma^2t)} + e^{-z^2/(2\sigma^2t)} + \frac{(\log L)^2}{L} \right) = O\left(  \frac{(e^{-2x}\vee 1)(\log L)^2}{L}   \right), 
\]
and otherwise, we have $\alpha > 0$ so that for any $\eps \in (0,1)$,  
\begin{align*}
	\sC_t(\eta_L,z) &= O\left( L\log L \cdot e^{-(1-\eps)z^2/(\sigma^2t)} + e^{-(1-\eps)z^2/(2\sigma^2t)} + \frac{(\log L)^2}{L} + e^{-\alpha\eps z} \right) \\
	&= O\left( \frac{(e^{-2(1-\eps)x} \vee 1)(\log L)^2}{L^{1-2\eps}} + e^{-\alpha\eps\sqrt{t\log L}} \right). \qedhere
\end{align*}
\end{proof}

We now prove the result referenced in the previous theorem. Lemma \ref{LCovrate} is the $L$-step equivalent of Lemma \ref{covlemmafull}. As we will see in its proof, where in Lemma \ref{covlemmafull} we have $s^{2/3}$ here we will use the value $2\sigma \sqrt{s\log s}$, where the constants appear for technical reasons that will hopefully become clear.

\begin{lemma}\label{LCovrate} If $z \sim c\sqrt{t\log L}$ as $t\to\infty$ for some $c > 0$ and $L \to \infty$ such that $L = o(\sqrt{t/\log t})$ as $t\to\infty$, then 
\begin{align*}
	&\sum_{k\in\Z} \int_0^t \left[P_0(\xi_s=k)-P_0(\xi_s=k+L+1)\right]^2P_0(\xi_{t-s} \ge z-k)^2\,ds \\
	&= O\left( L\log L \cdot e^{-z^2/(\sigma^2t)} + e^{-z^2/(2\sigma^2t)} + \frac{(\log L)^2}{L} \right), \qquad t \to\infty. 
\end{align*}
\end{lemma}

\begin{proof}
Define $\alpha(y) = \sigma\sqrt{4y\log y}$. Here, $\alpha$ is increasing on $y > 1/e$, and so $\alpha^{-1}(y)$ exists for large $y$. 
Write
\begin{align} \nonumber
	&\sum_{k\in \Z} \int_0^t \left[P_0(\xi_s=k)-P_0(\xi_s = k+L+1)\right]^2P_0(\xi_{t-s} \ge z-k)^2\,ds \\ 
	 \label{eq:Lpt2}
	&\le
	 \int_{L^2}^{t-z^{4/3}} \sum_{|k|\le \alpha(s) \wedge z}  \left[P_0(\xi_s=k)-P_0(\xi_s = k+L+1)\right]^2P_0(\xi_{t-s} \ge z-k)^2 \,ds  \\  \label{eq:Lpt5}
	&\quad +  \int_0^{L^2}  \sum_{|k| \le \alpha(s) \wedge z}  \left[P_0(\xi_s=k)-P_0(\xi_s = k+L+1)\right]^2P_0(\xi_{t-s} \ge z-k)^2 \,ds \\  \label{eq:Lpt6}
	&\quad +  \int_0^{L^2}  \sum_{|k| > \alpha(s)}  \left[P_0(\xi_s=k)-P_0(\xi_s = k+L+1)\right]^2P_0(\xi_{t-s} \ge z-k)^2 \,ds \\ \label{eq:Lpt3}
	&\quad + \int_{L^2}^{t} \sum_{|k| > \alpha(s)}  \left[P_0(\xi_s=k)-P_0(\xi_s = k+L+1)\right]^2 \,ds \\ \label{eq:Lpt4}
	&\quad + \int_{\alpha^{-1}(z)}^t \sum_{z < |k| \le \alpha(s)}  \left[P_0(\xi_s=k)-P_0(\xi_s = k+L+1)\right]^2 \,ds \\
	\label{eq:Lpt1}
	&\quad +  \int_{t-z^{4/3}}^t\sum_{|k|\le \alpha(s) \wedge z} \left[P_0(\xi_s=k)-P_0(\xi_s = k+L+1)\right]^2P_0(\xi_{t-s} \ge z-k)^2\,ds.
\end{align}
Each part above is bounded in seperate steps.

{\bf Step 1.} We show that \eqref{eq:Lpt2} is of order 
\[
	L\log L \cdot e^{-z^2/(\sigma^2t)} + \frac{\log L}{L^2}. 
\]
When $s < t - z^{4/3}$ and $k \le z$, 
\begin{align*}
	 &
	 \left|P_0(\xi_s=k)-P_0(\xi_s = k+L+1)\right|P_0(\xi_{t-s} \ge z-k) \\
	 &=
	 \left|P_0(\xi_s=k)-P_0(\xi_s = k+L+1)\right| \exp\left( - \frac{(z-k)^2}{2\sigma^2(t-s)} + O(1)\right) \\
	 &\le 
	 \left| \frac{e^{-k^2/(2\sigma^2s)}}{\sqrt{2\pi s}} - \frac{e^{-(k+L+1)^2/(2\sigma^2s)}}{\sqrt{2\pi s}} \right| \exp\left( - \frac{(z-k)^2}{2\sigma^2(t-s)} + O(1)\right)  \\
	 &\qquad 
	 +
	 \left| \Delta_s(k+L+1) - \Delta_s(k)\right| , 
\end{align*}
where $\Delta_t(k)$ is defined as in Lemma \ref{GradDiff}. Using that lemma, note that 
\begin{align*}
	\sum_{|k| \le \alpha(s)\wedge z} \left| \Delta_s(k+L+1) - \Delta_s(k)\right| 
	&\le \frac{C(L+1)}{s^2} \cdot 2\sigma\sqrt{4s\log s} = O\left( \frac{L\sqrt{\log s}}{s^{3/2}}\right). 
\end{align*}
Next, 
\begin{align*}
	&\left| \frac{e^{-k^2/(2\sigma^2s)}}{\sqrt{2\pi s}} - \frac{e^{-(k+L+1)^2/(2\sigma^2s)}}{\sqrt{2\pi s}} \right| = \frac{e^{-k^2/(2\sigma^2s)}}{\sqrt{2\pi s}} \left|1 - e^{-L^2/(2\sigma^2s) - Lk/(\sigma^2s)}\right|,
\end{align*}
and when $|k| \le \alpha(s) =O(\sqrt{s \log s})$, $|k|^n/s^n = O(|k|/s)$ for all $n$, so 
\begin{align*}
	\left|1 - e^{-L^2/(2\sigma^2s) - Lk/(\sigma^2s)}\right| = O\left(\frac{L^2}{s} + \frac{L|k|}{s} \right). 
\end{align*}
It follows that 
\begin{align*}
	&\sum_{|k|\le \alpha(s) \wedge z}  \left[P_0(\xi_s=k)-P_0(\xi_s = k+L+1)\right]^2P_0(\xi_{t-s} \ge z-k)^2 \\
	&= \sum_{|k|\le\alpha(s)\wedge z} \left( \frac{L^4}{s^{3}} + \frac{L^2k^2}{s^{3}}\right) \exp\left( -\frac{k^2}{\sigma^2s} - \frac{(z-k)^2}{\sigma^2(t-s)} + O(1)\right) + O\left( \frac{L^2\log s}{s^{3}}\right) \\
	&= e^{-z^2/(\sigma^2t)} \left(\frac{L^4}{s^{3}} + \frac{L^2\log s}{s^{2}} \right) \cdot O\left( \int_{-\infty}^\infty e^{-u^2/(\sigma^2s)}\,du \right) + O\left( \frac{L^2\log s}{s^{3}}\right) \\
	&= O\left( e^{-z^2/(\sigma^2t)}\left( \frac{L^4}{s^{5/2}} + \frac{L^2\log s}{s^{3/2}}\right)  + \frac{L^2\log s}{s^{3}}\right). 
\end{align*}
It follows that for some constant $C$, 
\begin{align*}
	&\int_{L^2}^{t-z^{4/3}} \sum_{|k| \le \alpha(s) \wedge z}  \left[P_0(\xi_s=k)-P_0(\xi_s = k+L+1)\right]^2P_0(\xi_{t-s} \ge z-k)^2 \,ds  \\
	&\le C \int_{L^2}^{t} \left[ e^{-z^2/(\sigma^2t)}\left( \frac{L^4}{s^{5/2}} + \frac{L^2\log s}{s^{3/2}}\right)  + \frac{L^2\log s}{s^{3}}\right]\,ds
	= O\left(L\log L \cdot e^{-z^2/(\sigma^2t)} + \frac{\log L}{L^2} \right). 
\end{align*}

{\bf Step 2.} \eqref{eq:Lpt5} is of order $L\log L \cdot e^{-z^2/(\sigma^2t)}$. 
To see this, first note
\begin{align*}
	&\int_0^{L^2} \sum_{|k| \le \alpha(s)\wedge z} \left[P_0(\xi_s = k) - P_0(\xi_s = k+L+1)\right]^2P_0(\xi_{t-s} \ge z-k)^2\,ds \\
	&\le \int_0^{L^2} P_0(\xi_{t-s} \ge z)^2 \sum_{k \le 0} \left[P_0(\xi_s = k) - P_0(\xi_s = k+L+1)\right]^2\,ds \\
	&\quad + \int_0^{L^2}  \sum_{0 < k \le \alpha(s)\wedge z} \left[P_0(\xi_s = k) - P_0(\xi_s = k+L+1)\right]^2P_0(\xi_{t-s} \ge z-k)^2\,ds.
\end{align*}
From Lemma \ref{RWchern}, 
\[
	P_0(\xi_{t-s} \ge z)^2 = \exp\left( -\frac{z^2}{\sigma^2(t-s)} + O\left(\frac{z^4}{(t-L^2)^3}\right)\right) = O\left(e^{-z^2/(\sigma^2t)}\right), 
\]
and using $\sup_l P_0(\xi_s = l) \le Cs^{-1/2}$ for some $C > 0$, 
\begin{align*}
	\sum_{k \in \Z} \left[P_0(\xi_s = k) - P_0(\xi_s = k+L+1)\right]^2 &\le \frac{C}{\sqrt{s}} \sum_{k\in \Z} \left| P_0(\xi_s = k) - P_0(\xi_s = k+L+1) \right|,
\end{align*}
which is of order $L/s$ by Lemma \ref{PointDiff}. Then for some (possibly different) $C > 0$, 
\begin{align}\nonumber
	&\int_0^{L^2} P_0(\xi_{t-s} \ge z)^2 \sum_{k \le 0} \left[P_0(\xi_s = k) - P_0(\xi_s = k+L+1)\right]^2\,ds \\ \nonumber
	&\le e^{-z^2/(\sigma^2t)} \int_0^1 \left[ P_0(\xi_s \le 0) + P_0(\xi_s \le L+1)\right]^2\,ds +C e^{-z^2/(\sigma^2t)} \int_1^{L^2} \frac{L}{s}\,ds \\ \label{eq:negk}
	&= O\left(L\log L \cdot e^{-z^2/(\sigma^2t)}\right). 
\end{align}
Similarly, since $z \sim c\sqrt{t\log L}$ and $L = o(\sqrt{t/\log t})$ as $t \to \infty$, $z - \alpha(L^2) > 0$ for large enough $t$, and so 
\begin{align*}
	&\int_0^{L^2}  \sum_{0 < k \le \alpha(s)\wedge z} \left[P_0(\xi_s = k) - P_0(\xi_s = k+L+1)\right]^2P_0(\xi_{t-s} \ge z-k)^2\,ds \\
	&\le 4\int_0^1 P_0(\xi_{t-s} \ge z - \alpha(1))^2\,ds + C\int_1^{L^2} \frac{L}{s} P_0(\xi_{t-s} \ge z - \alpha(L^2))^2\,ds \\
	&= 4\exp\left(-\frac{z^2}{\sigma^2t} + O(1)\right) \\
	&\qquad + L\log L^2 \exp\left( - \frac{\left(z - 2\sigma L\sqrt{\log L^2}\right)^2}{\sigma^2t} + O\left(\frac{z^4 + L^4}{(t-L^2)^3}\right) \right) \\
	&= O\left(L\log L \cdot e^{-z^2/(\sigma^2t)} \right), 
\end{align*}
where we again used Lemma \ref{RWchern}, and also that $z\sim c \sqrt{t\log L}$ and $L = o(\sqrt{t/\log t})$ imply 
\[
	\frac{z L \sqrt{\log L}}{t} \sim \frac{c L\log L}{\sqrt{t}} = O(1). 
\]

{\bf Step 3.} We show that \eqref{eq:Lpt6} is of order $L\log L \cdot e^{-z^2/(\sigma^2 t)} + e^{-z^2/(2\sigma^2t)}$. 
Again break up $k < 0$ and $k > 0$ and use the bound in \eqref{eq:negk}:  
\begin{align*}
	&\int_0^{L^2}  \sum_{|k| > \alpha(s)}  \left[P_0(\xi_s=k)-P_0(\xi_s = k+L+1)\right]^2P_0(\xi_{t-s} \ge z-k)^2\\
	&= O\left(L\log L \cdot e^{-z^2/(\sigma^2t)}\right) + \int_0^{L^2}  \sum_{k > \alpha(s)}  \left[P_0(\xi_s=k)-P_0(\xi_s = k+L+1)\right]^2P_0(\xi_{t-s} \ge z-k)^2. 
\end{align*}
By the Markov property and Lemma \ref{RWchern},  
\begin{align*}
	&\int_0^{L^2} \sum_{k > \alpha(s)}  \left[P_0(\xi_s=k)-P_0(\xi_s = k+L+1)\right]^2P_0(\xi_{t-s} \ge z-k)^2\\
	&\le \int_0^{L^2} P_0(\xi_s > \alpha(s)) \left[ \sum_{k \in \Z} P_0(\xi_s = k)P_k(\xi_{t-s} \ge z) + \sum_{k\in \Z} P_0(\xi_s = k)P_k(\xi_{t-s} \ge z - L - 1)\right]\,ds \\
	&\le \left[ P_0(\xi_t \ge z) + P_0(\xi_t \ge z-L-1)\right] \left[1 + \int_1^\infty \exp\left( - 2\log s + O(1) \right) \,ds \right] = O\left( e^{-z^2/(2\sigma^2t)}\right), 
\end{align*}
since $(z-L-1)^2/t = z^2/t + o(1)$ and $z \sim c\sqrt{t\log L}$. 

{\bf Step 4.} We show that \eqref{eq:Lpt3} is of order $L^{-5}$. This is done like in Step 3 of of the proof of Lemma \ref{covlemmafull} using Lemma \ref{RWchern}:
\begin{align*}
	&\int_{L^2}^{t} \sum_{|k| > \alpha(s)}  \left[P_0(\xi_s=k)-P_0(\xi_s = k+L+1)\right]^2P_0(\xi_{t-s}\ge z-k)^2 \,ds \\
	&\le 2\int_{L^2}^\infty P_0\left(\xi_s > \sigma\sqrt{4s\log s}\right)^2\,ds \\
	&\quad + \int_{L^2}^\infty \left[ P_0\left(\xi_s > \sigma\sqrt{4s\log s} + L+1\right) + P_0\left(\xi_s < L+1 - \sigma\sqrt{4s\log s}\right)\right]^2\,ds \\
	&\le 6\int_{L^2}^\infty P_0\left(\xi_s > \sigma\sqrt{4s\log s} - L-1\right)^2\,ds \\
	&= 6\int_{L^2}^\infty \exp\left( - 4\log s + \frac{4 L\sqrt{\log s}}{\sigma \sqrt{s}} + O(1)\right)\,ds = O\left(\frac{1}{L^5} \right), \qquad t \to \infty, 
\end{align*}
since for large enough $t$, $L > e^{32}$, hence for $s > L^2$, 
\[
	\exp\left(\frac{4 L\sqrt{\log s}}{\sigma \sqrt{s}}\right) \le \exp\left( 4\sqrt{2\log L}\right) = \exp\left(\frac{4\sqrt{2}}{\sqrt{\log L}} \cdot \log L\right) = O(L), \qquad t\to\infty. 
\]

{\bf Step 5.} \eqref{eq:Lpt4} is of order 
\[
	\frac{L^2(\log t)^{3/2}e^{-z^2/(\sigma^2t)}}{\sqrt{t\log L}}
\]
as $t \to \infty$. 
By an argument like in Step 1 using Lemma \ref{GradDiff}, for some constant $C$, 
\begin{align*}
	&\sum_{z < |k| \le \alpha(s) }  \left[P_0(\xi_s=k)-P_0(\xi_s = k+L+1)\right]^2 \\
	&\le C \left[ \sum_{z < |k|\le \alpha(s)} e^{-k^2/(\sigma^2s)} \left(\frac{L^2}{s^{3}} + \frac{L^2k^2}{s^{3}}\right) + \frac{L^2\log s}{s^{3}} \right] \\
	&\le C \left[ \left(\frac{L^4}{s^3} + \frac{L^2\log s}{s^{2}}\right)\int_z^\infty e^{-u^2/(\sigma^2s)}\,du + \frac{L^2\log s}{s^{3}} \right], 
\end{align*}
where $\int_z^\infty e^{-u^2/(\sigma^2s)}\,du \le \sigma\sqrt{s/2} e^{-z^2/(\sigma^2s)}$. 
Then, for some possibly different constant $C$, 
\begin{align}\nonumber
	& \int_{\alpha^{-1}(z)}^t \sum_{z < |k| \le \alpha(s) }  \left[P_0(\xi_s=k)-P_0(\xi_s = k+L+1)\right]^2 \,ds \\ \nonumber
	&\le C \int_{\alpha^{-1}(z)}^t \left[ e^{-z^2/(\sigma^2s)} \left( \frac{L^4}{s^{5/2}} + \frac{L^2\log s}{s^{3/2}}\right) + \frac{L^2\log s}{s^{3}}\right]\,ds \\ 
	\label{eq:theabove}
	&=O\left(\frac{L^4e^{-z^2/(\sigma^2t)}}{(\alpha^{-1}(z))^{3/2}} + \frac{L^2e^{-z^2/(\sigma^2t)}\log \alpha^{-1}(z)}{\sqrt{\alpha^{-1}(z)}} + \frac{L^2\log \alpha^{-1}(z)}{(\alpha^{-1}(z))^{2}}\right). 
\end{align}
Now, if $\beta(y) = y\log y$, then for $y > 1/e$, $\beta^{-1}(y) = y/W(y)$, where 
\[
	W(y) = \log y - \log\log y + o(1), \qquad y \to \infty, 
\]
is the Lambert $W$ function (see, for example \cite{CorGonHarJefKnu}). Further, by definition $\alpha^{-1}(z) = \beta^{-1}(z^2/(4\sigma^2))$ and 
$\log \alpha^{-1}(z) = z^2/(4\sigma^2\alpha^{-1}(z))$. It follows that 
\[
	\frac{1}{\alpha^{-1}(z)} = \frac{4\sigma^2}{z^2}\left(\log \left(\frac{z^2}{4\sigma^2}\right) - \log \log \left(\frac{z^2}{4\sigma^2}\right) + o(1) \right) = O\left(\frac{\log z}{z^2}\right), \qquad z \to \infty,  
\]
and \eqref{eq:theabove} is of order 
\begin{align*}
	&\frac{L^4 e^{-z^2/(\sigma^2t)}(\log z)^{3/2}}{z^3} + \frac{L^2e^{-z^2/(\sigma^2t)}(\log z)^{3/2}}{z} + \frac{L^2(\log z)^3}{z^4} 
	= O\left(\frac{L^2(\log t)^{3/2}e^{-z^2/(\sigma^2t)}}{\sqrt{t\log L}}\right).
\end{align*}

{\bf Step 6.} We show that \eqref{eq:Lpt1} is of order 
\[
	\frac{L^2(\log t) e^{-z^2/(\sigma^2 t)}}{\sqrt{t}} + \frac{(\log L)^2}{L},
\]
which is argued similarly to as in Step 5 of Lemma \ref{covlemmafull}. 
First write, for some $\eps \in (0,1)$,  
\begin{align*}
	& \int_{t-z^{4/3}}^t\sum_{|k|\le \alpha(s) \wedge z} \left[P_0(\xi_s=k)-P_0(\xi_s = k+L+1)\right]^2P_0(\xi_{t-s} \ge z-k)^2\,ds \\
	&\le \int_{t-z^{4/3}}^t\sum_{|k|\le \eps z} \left[P_0(\xi_s=k)-P_0(\xi_s = k+L+1)\right]^2P_0(\xi_{t-s} \ge z-k)^2\,ds  \\
	&\quad + \int_{t-z^{4/3}}^t\sum_{\eps z < |k|\le \alpha(s)} \left[P_0(\xi_s=k)-P_0(\xi_s = k+L+1)\right]^2\,ds.  
\end{align*}
Then by Lemma \ref{PointDiff}, 
\begin{align*}
	&\int_{t-z^{4/3}}^t\sum_{|k|\le \eps z} \left[P_0(\xi_s=k)-P_0(\xi_s = k+L+1)\right]^2P_0(\xi_{t-s} \ge z-k)^2\,ds \\
	&\le \frac{C(L+1)}{(1-\eps)^4z^{4/3}} \int_{t-z^{4/3}}^t \frac{ds}{s} = \frac{C(L+1)}{(1-\eps)^4z^{4/3}} \log\left(\frac{t}{t-z^{4/3}}\right) = O\left(\frac{L}{(1-\eps)^4t}\right). 
\end{align*}
Furthermore, by the arguments in Steps 1 and 3, 
\begin{align*}
	&\sum_{\eps z < |k| \le \alpha(s)} \left[P_0(\xi_s=k) - P_0(\xi_s=k+L+1)\right]^2 \\
	&= \left( \frac{L^4}{s^{5/2}} + \frac{L^2\log s}{s^{3/2}} \right) e^{-\eps^2z^2/(\sigma^2 t)} + O\left(\frac{L^2\log s}{s^2}\right), 
\end{align*}
and so 
\begin{align*}
	&\int_{t-z^{4/3}}^t\sum_{\eps z < |k|\le \alpha(s)} \left[P_0(\xi_s=k)-P_0(\xi_s = k+L+1)\right]^2\,ds \\
	&= O\left( e^{-\eps^2z^2/(\sigma^2 t)} \left( \frac{L^4}{t^{3/2}} + \frac{L^2\log t}{\sqrt{t}} \right) + \frac{L^2\log t}{t^2}\right) = O\left(  \frac{L^2\log t \cdot e^{-\eps^2z^2/(\sigma^2 t)}}{\sqrt{t}} + \frac{L^2\log t}{t^2}\right). 
\end{align*}
Now, note that $L = o(\sqrt{t}\log L)$ and let 
$\eps = 1 - \frac{\sqrt{L}}{t^{1/4}\sqrt{\log L}}$. 
Then, $(1 - \eps)^4 = L^2/(t\log^2 L)$, so 
$\frac{L}{(1-\eps)^4t} = \frac{(\log L)^2}{L}$. 
Furthermore, 
\[
	\exp\left(-\frac{\eps^2z^2}{\sigma^2 t}\right) \le \exp\left( - \frac{z^2}{\sigma^2 t} \left(1 + \frac{L}{\sqrt{t}\log L}\right)\right) = \exp\left(-\frac{z^2}{\sigma^2 t}(1+o(1))\right) = O\left(e^{-z^2/(\sigma^2 t)}\right), 
\]
since $z \sim c\sqrt{t \log L}$. Finally, note that $L^2\log t/t^2 = o(1) \cdot (\log L)^2/L$ to obtain the result. 

Steps 1--6 together prove the lemma. 
\end{proof}

\appendix
\section{Appendix}

Here we provide some supplementary lemmas used in proofs. Versions of some of these results for discrete-time random walks are available elsewhere, for example in \cite{Bor09,LawLim}. For convenience of the reader, we provide details. 

\begin{lemma}\label{flemma} Let $f(u) = E[(X-u)^+] = \varphi(u) - uP(X>u)$, where $X \sim \sN(0,1)$ and $\varphi$ denotes the standard normal density function. Then for $u \ge 1$ and $v \ge 0$, 
\begin{align*}
	&f(u) =  \frac{\varphi(u)}{u^2} + O\left(\frac{\varphi(u)}{u^4}\right), \qquad\mbox{and}\\
	&f(u) - f(u+v) = \frac{\varphi(u)}{u^2} \left(1 - \frac{e^{-v^2/2-uv}}{(1 +v/u)^2}\right) + O\left(\left(\frac{v}{u^3} \wedge \frac{1}{u^4}\right) \varphi(u) \right). 
\end{align*}
\end{lemma}

\begin{proof} We show the result for $f(u)-f(u+v)$; the other, which is more standard, follows from letting $v \to \infty$. Integrating by parts several times in the usual way, we compute 
\[
	P(X > u) = \int_u^\infty \varphi(s)\,ds = \varphi(u) \left( \frac{1}{u} - \frac{1}{u^3}\right) + 3\int_u^\infty \frac{\varphi(s)}{s^4}\,ds. 
\]
It follows that 
\begin{align*}
	f(u) - f(u+v) &= \frac{\varphi(u)}{u^2} - \frac{\varphi(u+v)}{(u+v)^2} - 3u\int_u^\infty \frac{\varphi(s)}{s^4}\,ds + 3(u+v)\int_{u+v}^\infty \frac{\varphi(s)}{s^4}\,ds \\
	&= \frac{\varphi(u)}{u^2} \left(1 - \frac{e^{-v^2/2-uv}}{(1 +v/u)^2}\right) - 3u\int_{u}^{u+v}\frac{\varphi(s)}{s^4}\,ds + 3v\int_{u+v}^\infty \frac{\varphi(s)}{s^4}\,ds. 
\end{align*}
Now note that 
\begin{align*}
	u\int_{u}^{u+v}\frac{\varphi(s)}{s^4}\,ds &= \frac{u}{\sqrt{2\pi}} \int_{u}^{u+v} \frac{e^{-s^2/s}}{s^4}\,ds \le \frac{v}{\sqrt{2\pi} u^3} e^{-u^2/2} = \frac{v\varphi(u)}{u^3}, 
\end{align*}
and
\begin{align*}
	v\int_{u+v}^\infty \frac{\varphi(s)}{s^4}\,ds &\le \frac{v}{(u+v)^4} P(X > u) = O\left(\frac{v\varphi(u)}{u^5}\right), 
\end{align*}
while on the other hand, since $u+v \ge u \vee v$, 
\[
	u\int_{u}^{u+v}\frac{\varphi(s)}{s^4}\,ds + v\int_{u+v}^\infty \frac{\varphi(s)}{s^4}\,ds \le \frac{P(X>u+v)}{u^3} +  \frac{vP(X>u)}{v\cdot u^3} = O\left(\frac{\varphi(u)}{u^4}\right). 
\]
combining the above bounds gives the result. 
\end{proof}

The next result is the Corollary stated on page 552 of \cite[XVI.7]{Fel71} with respect to an expansion in a central limit theorem. Note that by symmetry of $\{p_i\}$, Condition \ref{mgf} implies that the moment generating function is finite in a neighborhood of $0$. 

\begin{lemma}\label{FellerNormal} Let $\xi_t$ be a random walk such that $\xi_1$ has finite moment generating function in a neighborhood of the origin. If $x \to \infty$ such that $x=o(t^{1/6})$ as $t \to \infty$, then 
\begin{equation}\nonumber \label{eq:FellerLD}
	\left|1 - \frac{P_0(\xi_t > \sigma x\sqrt{t})}{P(X>x)}\right| = O\left(\frac{x^3}{\sqrt{t}}\right), \qquad t\to\infty. 
\end{equation}
\end{lemma}

For the next four results, let $\{\xi_t\}$ be a symmetric, continuous-time random walk with $\xi_0 = 0$, and let $\sigma^2 = E[\xi_1^2]$. Lemma \ref{RWchern} below gives a non-asymptotic, Chernoff-like bound for tails of $\xi_t$ that is valid in a larger regime than the previous result. 

\begin{lemma}\label{RWchern} Suppose the jump distribution $\{p_i\}$ of $\xi_t$ satisfies Condition \ref{mgf} for $\theta > 0$.
Then, for any $0 \le x \le \sigma^2\theta t$, 
\[
	P(\xi_t \ge x) = \exp\left( -\frac{x^2}{2\sigma^2 t} + O\left(\frac{x^4}{t^3}\right)\right). 
\]
\end{lemma}

\begin{proof} Since $\{p_i\}$ is symmetric, $E[\xi_1^n] = 0$ for $n$ odd. For every $t \ge 0$ and $0<\lambda\le\theta$, we can use the representation of $\xi_t$ as a discrete-time random walk $S_n$ with a Poisson$(t)$ number of steps to compute 
\begin{equation}\label{eq:contimemgf}
	E\big[e^{\lambda \xi_t}\big] = \sum_{n\ge0} \frac{e^{-t} t^n}{n!} \left( E\big[e^{\lambda S_1}\big]\right)^n = \exp\left(- t - tE\big[e^{\lambda S_1}\big]\right) = \left(E\big[e^{\lambda \xi_t}\big]\right)^t. 
\end{equation}
Then, 
\begin{align*}
	\log P(\xi_t \ge x) &\le -\lambda x + t\log E\left[e^{\lambda \xi_1}\right] \\
	&= -\lambda x + t\log \left( 1 + \frac{\lambda^2\sigma^2}{2} + O\left( \lambda^4\right) \right) = -\lambda x + \frac{\lambda^2\sigma^2 t}{2} + O\left(\lambda^4 t\right). 
\end{align*}
Choosing $\lambda = x/(\sigma^2 t) \le \theta$ gives the result.
\end{proof}

The next lemma is a local central limit theorem and an extention of Theorem 2.5.6 in \cite{LawLim} for simple random walks to a more general jump distribution $\{p_i\}$. It follows along the same lines, but we sketch the proof details. 

\begin{lemma}\label{localCLT} 
There exists $\eps \in (0,1)$ such that if $|x| < \eps t$ , then 
\[
	P(\xi_t = x) = \frac{e^{-x^2/(2\sigma^2t)}}{\sqrt{2\pi t}} \exp\left( O\left(\frac{1}{\sqrt{t}} + \frac{|x|^3}{t^2} \right)\right). 
\]
\end{lemma}

\begin{proof} 
Let $\{S_n\}$ denote a discrete-time random walk with the same transition probabilities as $\xi_t$. Then by Theorem 2.3.11 in \cite{LawLim}, there exists $\rho > 0$ such that when $|x| < \rho n$, 
\[
	P(S_n = x) = \frac{e^{-x^2/(2\sigma^2n)}}{\sqrt{2\pi n}} \exp\left(O\left( \frac{1}{n} + \frac{|x|^4}{n^3} \right)\right). 
\]
Let $\delta \in (0,1/2)$. If $\{N_t\}$ is a rate-$1$ Poisson process, then 
\begin{align*}
	P(\xi_t = x) 
	&\le \sum_{|n-t| \le \delta t} P(N_t = n)P(S_n=x) + P(|N_t - t| > \delta t) \\
	&= \sum_{|n-t| \le \delta t} P(N_t = n)P(S_n=x) + O(e^{-\beta t})
\end{align*}
for some $\beta > 0$. From \cite[Prop. 2.5.5]{LawLim}, when $|n-t| \le t/2$, 
\[
	P(N_t = n) = \frac{e^{-(n-t)^2/(2t)}}{\sqrt{2\pi t}} \exp\left(O\left(\frac{1}{\sqrt{t}} + \frac{|n-t|^3}{t^2} \right)\right). 
\]
Let $\eps \le \rho/2$, so that $|x| < \eps t$ implies $|x| < (\rho/2)t < (1-\delta)\rho t \le \rho n$ when $|n-t| \le \delta t$. Then, 
\begin{align*}
	&\sum_{|n-t| \le \delta t} P(N_t = n)P(S_n=x) \\
	&= \frac{e^{-x^2/(2\sigma^2t)}}{\sqrt{2\pi t}} \exp\left( O\left( \frac{1}{\sqrt{t}} + \frac{|x|^3}{t^2} \right)\right) \sum_{|n-t| \le \delta t} \frac{e^{-\left(\frac{1}{2} - \frac{\eps}{4(1-\delta)} \right)\frac{(n-t)^2}{t}}}{\sqrt{(1-\delta)2\pi t}} \exp\left(O\left( \frac{|n-t|^3}{t^2} \right)\right),  
\end{align*}
where we used $(1-\delta)t \le n \le (1+\delta)t$ and 
\begin{align*}
	-\frac{x^2}{2n} &= -\frac{x^2}{2t} + \frac{x^2}{2t}\left(1 - \frac{t}{n}\right) \le -\frac{x^2}{2t} + \frac{x^2|n-t|}{2t(1-\delta)t} 
	= -\frac{x^2}{2t} + 
	 \frac{\eps|n-t|^2}{4(1-\delta)t} + O\left( \frac{|x|^3}{t^2}\right).
\end{align*} 
Now take $\delta$ small enough so that $O(|n-t|^3/t^2) \le (n-t)^2/(8t)$ whenever $|n-t| \le \delta t$ and subsequently select $\eps$ so that $\eps < \min\{1-\delta, \rho/2\}$. Then,  
\[
	-\left(\frac{1}{2} - \frac{\eps}{4(1-\delta)} \right)\frac{(n-t)^2}{t} \le - \frac{(n-t)^2}{4t}, 
\]
and 
\begin{align*}
	&\sum_{|n-t| \le \delta t} \frac{e^{-\left(\frac{1}{2} - \frac{\eps}{4(1-\delta)} \right)\frac{(n-t)^2}{t}}}{\sqrt{(1-\delta)2\pi t}} \exp\left(O\left( \frac{|n-t|^3}{t^2} \right)\right) \\
	&\le  \frac{e^{O(1)}}{\sqrt{\pi t}} \sum_{|n-t| \le t^{2/3}} e^{-(n-t)^2/(4t)}  +  \frac{1}{\sqrt{\pi t}} \sum_{|n-t| > t^{2/3}} e^{-(n-t)^2/(8t)} = 1 + O\left(\frac{1}{\sqrt{t}}\right) = \exp\left(O\left(\frac{1}{\sqrt{t}}\right)\right). 
\end{align*}
Combining with everything above completes the proof. 
\end{proof}

The estimates in Lemmas \ref{GradDiff} and \ref{PointDiff} are updates of the second part of Theorem 2.3.6 and Proposition 2.4.1 in \cite{LawLim} for discrete-time random walks to continuous time random walks in one dimension.

\begin{lemma}\label{GradDiff}  
Define $\Delta_t(x) = P(\xi_t = x) - f_t(x)$ for $f_t(x) = (2\pi \sigma^2 t)^{-1/2}e^{-x^2/(2\sigma^2 t)}$. 
There is a constant $C \in(0,\infty)$ such that for all $t > 0$ and $x, y\in \Z$, 
\[	
	\left| \Delta_t(x+y) - \Delta_t(x)\right| \le \frac{C|y|}{t^2}. 
\]
\end{lemma}

\begin{proof} It suffices to show the case when $y=1$ (and then apply the triangle inequality). 
We first claim that for each $\eps > 0$, there is a function $u(x,t)$ such that 
\begin{equation}\label{eq:charfuncclaim}
	P(\xi_t = x) = f_t(x) + u(x,t) + \frac{1}{2\pi\sqrt{t}}\int_{|s| \le \eps\sqrt{t}} e^{-ixs/\sqrt{t} - s^2/(2\sigma^2)}F_t(s)\,ds, 
\end{equation}
where, for some $\gamma , C> 0$, $|u(x,t)| \le Ce^{-\gamma t}$ and $|F_t(x)| \le C|x|^4t^{-1}$. This is Lemma 2.3.4 in \cite{LawLim} with $n$ replaced by $t$, for which the exact same proof holds using \eqref{eq:contimemgf}. 
Using \eqref{eq:charfuncclaim}, we have that 
\begin{align*}
	\left|\Delta_t(x+1) - \Delta_t(x)\right|&= \left| P(\xi_t = x+1) - f_t(x+1) - P(\xi_t = x) + f_t(x)\right| \\
	&\le \frac{1}{2\pi\sqrt{t}} \int_{|s| \le \eps\sqrt{t}} \left| e^{-i(x+1)s/\sqrt{t}} - e^{-ixs/\sqrt{t}}\right| e^{-s^2/(2\sigma^2)}|F_t(x)|\,ds +Ce^{-\gamma t}\\
	&= \frac{1}{2\pi\sqrt{t}} \int_{|s| \le \eps\sqrt{t}} \left| e^{-is/\sqrt{t}} - 1\right| e^{-s^2/(2\sigma^2)}|F_t(x)|\,ds +Ce^{-\gamma t}\\
	&\le \frac{C}{2\pi t^2} \int_{-\infty}^\infty |s|^5e^{-s^2/(2\sigma^2)}\,ds + Ce^{-\gamma t} = O\left(\frac{1}{t^2}\right). \qedhere
\end{align*}
\end{proof}

\begin{lemma}\label{PointDiff}
There is a constant $C \in(0,\infty)$ such that 
\[
	\sum_{x\in\Z} \left|P(\xi_t = x) - P(\xi_t = x+y)\right| \le \frac{C|y|}{\sqrt{t}}. 
\]
\end{lemma}

\begin{proof} Given Lemma \ref{GradDiff}, the proof follows line by line as with Proposition 2.4.1 in \cite{LawLim}. 
\end{proof}

\section*{Funding}  The authors were supported partially by grant ARO W911NF-181-0311.

\end{document}